\documentclass[twoside,a4paper,11pt]{amsart}

\usepackage[symbol]{footmisc}
\usepackage{epigraph}
\usepackage{geometry}
\usepackage{amsmath}
\usepackage{amssymb}
\usepackage{amsthm}
\usepackage{mathrsfs}
\usepackage{mathtools}
\usepackage{esint}
\usepackage{graphicx}
\usepackage{color}
\usepackage{hyperref}
\usepackage[english]{babel}

\newcommand{\N}{\mathbb{N}}

\newcommand{\R}{\mathbb{R}}
\renewcommand{\S}{\mathbb{S}}

\newcommand{\cH}{\mathcal{H}}

\newcommand{\cN}{\mathcal{N}}

\newcommand{\cP}{\mathcal{P}}
\newcommand{\cQ}{\mathcal{Q}}

\newcommand{\cW}{\mathcal{W}}

\newcommand{\sC}{\mathscr{C}}
\newcommand{\sF}{\mathscr{F}}

\newcommand{\bv}{\mathbf v}

\newcommand{\tV}{\tilde{V}}

\newcommand{\diam}{\mbox{\rm diam}}

\newcommand{\supp}{\mbox{\rm supp}}

\renewcommand{\div}{\mbox{\rm div}}

\newcommand{\wsc}{\overset{\star}{\rightharpoonup}}

\newcommand{\epty}{\emptyset}

\newcommand{\sm}{\setminus}
\newcommand{\lgl}{\langle}
\newcommand{\rgl}{\rangle}
\newcommand{\pa}{\partial}
\newcommand{\con}{\subset}

\newcommand{\res}{
	\,\raisebox{-.127ex}{\reflectbox{\rotatebox[origin=br]{-90}{$\lnot$}}}\,
} 

\newcommand{\ep}{\varepsilon} 
\newcommand{\ga}{\gamma}
\newcommand{\be}{\beta}

\newcommand{\de}{\delta}

\newcommand{\ze}{\zeta}

\newcommand{\ro}{\rho}
\newcommand{\si}{\sigma}
\newcommand{\te}{\theta}
\newcommand{\De}{\Delta}
\newcommand{\Ga}{\Gamma}
\newcommand{\La}{\Lambda}
\newcommand{\Si}{\Sigma}

\newcommand{\Om}{\Omega}
\newcommand{\vp}{\varphi}

\newcommand{\sff}{\mbox{\rm II}}

\theoremstyle{plain}
\newtheorem{thm}{Theorem}[section] 

\theoremstyle{plain}

\theoremstyle{plain}
\newtheorem{prop}[thm]{Proposition}

\theoremstyle{plain}
\newtheorem{lemma}[thm]{Lemma}

\theoremstyle{plain}
\newtheorem{cor}[thm]{Corollary}

\theoremstyle{definition}

\theoremstyle{definition}
\newtheorem{remark}[thm]{Remark}

\theoremstyle{definition}

\title[Connected surfaces with boundary minimizing the Willmore energy]{Connected surfaces with boundary\\ minimizing the Willmore energy}
\author{Matteo Novaga}
\address{Dipartimento di Matematica, Universit\`{a} di Pisa, Largo Bruno Pontecorvo 5, 56127 Pisa, Italy}
\email{matteo.novaga@unipi.it}
\author{Marco Pozzetta}
\address{Dipartimento di Matematica, Universit\`{a} di Pisa, Largo Bruno Pontecorvo 5, 56127 Pisa, Italy}
\email{pozzetta@mail.dm.unipi.it}
\date{\today}

\makeatletter
\newtheorem*{rep@theorem}{\rep@title}
\newcommand{\newreptheorem}[2]{%
	\newenvironment{rep#1}[1]{%
		\def\rep@title{#2 \ref{##1}}%
		\begin{rep@theorem}}%
		{\end{rep@theorem}}}
\makeatother
\newreptheorem{theorem}{Theorem}

\begin{document}

\begin{abstract}
	For a given family of smooth closed curves $\ga^1,...,\ga^\alpha\con\R^3$ we consider the problem of finding an elastic \emph{connected} compact surface $M$ with boundary $\ga=\ga^1\cup...\cup\ga^\alpha$. This is realized by minimizing the Willmore energy $\cW$ on a suitable class of competitors. While the direct minimization of the Area functional may lead to limits that are disconnected, we prove that, if the infimum of the problem is $<4\pi$, there exists a connected compact minimizer of $\cW$ in the class of integer rectifiable curvature varifolds with the assigned boundary conditions. This is done by proving that varifold convergence of bounded varifolds with boundary with uniformly bounded Willmore energy implies the convergence of their supports in Hausdorff distance. Hence, in the cases in which a small perturbation of the boundary conditions causes the non-existence of Area-minimizing connected surfaces, our minimization process models the existence of optimal elastic connected compact generalized surfaces with such boundary data.
	We also study the asymptotic regime in which the diameter of the optimal connected surfaces is arbitrarily large. Under suitable boundedness assumptions, we show that rescalings of such surfaces converge to round spheres. The study of both the perturbative and the asymptotic regime is motivated by the remarkable case of elastic surfaces connecting two parallel circles located at any possible distance one from the other.\\
	The main tool we use is the monotonicity formula for curvature varifolds (\cite{SiEX}, \cite{KuSc}) that we extend to varifolds with boundary, together with its consequences on the structure of varifolds with bounded Willmore energy.
\end{abstract}
\maketitle

\tableofcontents	

\vspace{-0.5cm} 

\noindent{\small\textbf{MSC Codes (2010):} 49Q20, 74E10, 49J40, 53A05, 49Q10.}

\vspace{0.2cm} 


\vspace{0.2cm}

\section{Introduction}

\subsection{The Willmore energy} Let $\vp:\Si\to\R^3$ be an immersion of a $2$-dimensional manifold $\Si$ with boundary $\pa \Si$ in the Euclidean space $\R^3$. We say that an immersion is smooth if it is of class $C^2$. In such a case we define the second fundamental form of $\vp$ in local coordinates as
\[
\sff_{ij}(p)=(\pa_{ij}\vp(p))^\perp,
\]
for any $p\in \Si\sm\pa \Si$, where $(\cdot)^\perp$ denotes the orthogonal projection onto $(d\vp(T_p\Si))^\perp$. Denoting by $g_{ij}=\lgl\pa_i\vp,\pa_j\vp\rgl$ the induced metric tensor on $\Si$ and by $g^{ij}$ the components of its inverse, we define the mean curvature vector by
\[
\vec{H}(p)=\frac12 g^{ij}(p)\sff_{ij}(p),
\]
for any $p\in \Si\sm\pa \Si$, where sum over repeated indices is understood. The normalization of $\vec{H}$ is such that the mean curvature vector of the unit sphere points inside the ball and it has norm equal to one. Denoting by $\mu_\vp$ the volume measure on $\Si$, we define the Willmore energy of $\vp$ by
\[
\cW(\vp)=\int_\Si |\vec{H}|^2\,d\mu_\vp.
\]
For an immersion $\vp:\Si\to\R^3$ we will denote by $co_\vp:\pa \Si\to \R^3$ the conormal field, i.e. the unit vector field along $\pa\Si$ belonging to $d\vp(T\Si)\cap \left(d\vp|_{\pa\Si}(T\pa\Si)\right)^\perp$ and pointing outside of $\vp(\Si)$.\\

\noindent The study of variational problems involving the Willmore energy has begun with the works of T. Willmore (\cite{Wi65}, \cite{WiRG}), in which he proved that round spheres minimize $\cW$ among every possible immersed compact surface without boundary. The Willmore energy of a sphere is $4\pi$. In \cite{Wi65} the author proposed his celebrated conjecture, claiming that the infimum of $\cW$ among immersed smooth tori was $2\pi^2$. Such conjecture (eventually proved in \cite{MaNeWC}) motivated the variational study of $\cW$ in the setting of smooth surfaces without boundary. In such setting many fundamental results have been achieved, and some of them (in particular \cite{SiEX}, \cite{KuSc}, and \cite{ScBP}) developed a very useful variational approach, that today goes under the name of Simon's ambient approach. Such method relies on the measure theoretic notion of varifold as a generalization of the concept of immersed submanifold. We remark that, more recently, an alternative and very powerful variational method based on a weak notion of immersions has been developed in \cite{RiAA}, \cite{RiLI}, and \cite{RiVP}.

\noindent Following Simon's approach, the concept of curvature varifold with boundary (\cite{Ma}, \cite{Hu}), considered as a good generalization of smooth immersed surfaces, will be fundamental in this work. Such notion is recalled in Appendix A. We will always consider integer rectifiable curvature varifolds with boundary, that we will usually call simply varifolds. Roughly speaking a rectifiable varifold is identified by a couple $\bv(M,\te_V)$, where $M\con \R^3$ is $2$-rectifiable and $\te_V:M\to\N_{\ge1}$ is locally $\cH^2$-integrable on $M$, and we think at it as a $2$-dimensional object in $\R^3$ whose points $p$ come with a weight $\te_V(p)$. We recall here that a $2$-dimensional varifold $V=\bv(M,\te_V)$ has weight measure $\mu_V=\te_V\cH^2\res M$, that is a Radon measure on $\R^3$; moreover it has (generalized) mean curvature vector $\vec{H}\in L^1_{loc}(\mu_V;\R^3)$ and generalized boundary $\si_V$ if
\[
\int \div_{TM} X\,d\mu_V=-2\int \lgl \vec H, X \rgl \,d\mu_V + \int X\,d\si_V \qquad\forall\,X\in C^1_c(\R^3;\R^3),
\]
where $\si_V$ is a Radon $\R^3$-valued measure on $\R^3$ of the form $\si_V=\nu_V\si$, with $|\nu_V|=1$ $\si$-ae and $\si$ is singular with respect to $\mu_V$; also $\div_{TM}X(p)=tr(P^\top\circ\nabla X(p))$ where $P^\top$ is the matrix corresponding to the projection onto $T_pM$, that is defined $\cH^2$-ae on $M$.\\
By analogy with the case of sooth surfaces, we define the Willmore energy of a varifold $V=\bv(M,\te_V)$ by setting
\[
\cW(V)=\int |\vec{H}|^2\,d\mu_V \,\,\in[0,+\infty],
\]
if $V$ has generalized mean curvature $\vec{H}$, and $\cW(V)=+\infty$ otherwise.\\
A rectifiable varifold $V=\bv(M,\te_V)$ defines a Radon measure on $G_2(\R^3):=\R^3\times G_{2,3}$, where $G_{2,3}$ is the Grassmannian of $2$-subspaces of $\R^3$, identified with the metric space of matrices corresponding to the orthogonal projection on such subspaces. More precisely for any $f\in C^0_c(G_2(\R^3))$ we define
\[
V(f):=\int_{G_2(\R^3)} f(p,P)\,dV(p,P)=\int_{\R^3} f(p,T_pM) \,d\mu_V(p).
\]
In this way a good notion of convergence in the sense of varifolds is defined, i.e. we say that a sequence $V_n=\bv(M_n,\te_{V_n})$ converges to $V=\bv(M,\te_V)$ as varifolds if
\[
V_n(f)\to V(f),
\]
for any $f\in C^0_c(G_2(\R^3))$.\\

\noindent More recently, varifolds with boundary and Simon's method have been used also in the study of variational problems in the presence of boundary conditions. A seminal work is \cite{ScBP}, in which the author constructs branched surfaces with boundary that are critical points of the Willmore energy with imposed clamped boundary conditions, i.e. with fixed boundary curve and conormal field. Another remarkable work is \cite{Ei19}, in which an analogous result is achieved in the minimization of the Helfich energy. We also mention \cite{Po19}, in which the minimization problem of the Willmore energy of surfaces with boundary with fixed topology is considered, and the only constraint is the boundary curve, while the conormal is free, yielding the so-called natural Navier boundary condition.

\medskip

\subsection{Elastic surfaces with boundary}

\noindent If $\ga=\ga^1\cup...\cup\ga^\alpha$ is a finite disjoint union of smooth closed compact embedded curves, a classical formulation of the Plateau's problem with datum $\ga$ may be to solve the minimization problem
\begin{equation}
\min\left\{ \mu_\vp(\Si) \,\,|\,\, \vp:\Si\to\R^3, \, \vp|_{\pa\Si}:\pa\Si\to\ga \mbox{ embedding} \right\},
\end{equation}
that is one wants to look for the surface of least area having the given boundary. From a physical point of view, solutions of the Plateau's problem are good models of soap elastic films having the given boundary (\cite{Morgan}). Critical points of the Plateau's problem are called minimal surfaces and they are characterized by having zero mean curvature (this is true also in the non-smooth context of varifolds in the appropriate sense, see \cite{SiGMT}). In particular, minimal surfaces or varifolds with vanishing mean curvature have zero Willmore energy. However, as we are going to discuss, the Plateau's problem, and more generally the minimization of the Area functional, may be incompatible with some constraints, such as a connectedness constraint.\\
\noindent In this paper we want to study the minimization of the Willmore energy of varifolds $V$ with given boundary conditions, i.e. both conditions of clamped or natural type on the generalized boundary $\si_V$, adding the constraint that the support of the varifold must connect the assigned curves $\ga^1,...,\ga^\alpha$. Hence the minimization problems we will study have the form
\begin{equation}\label{P}
	\cP\,:=\, \min\left\{ \,\cW(V) \quad|\quad V=\bv(M,\te_V):\quad
	\si_V=\si_0,\quad \supp V\cup \ga \,\, \mbox{ compact, connected } \right\},
\end{equation}
for some assigned vector valued Radon measure $\si_0$, or
\begin{equation}\label{Q}
	\cQ\,:=\, \min\left\{ \,\cW(V) \quad|\quad V=\bv(M,\te_V):\quad
	|\si_V|\le \mu,\quad \supp V\cup \ga \,\, \mbox{ compact, connected } \right\},
\end{equation}
for some assigned positive Radon measure $\mu$ with $\supp\mu=\ga$.\\
Let us introduce a remarkable particular case that motivates our study. Let $\sC=[0,1]^2/_\sim$ be a cylinder. Let $R\ge1$ and $h>0$. We define
\begin{equation*}
\Ga_{R,h}:=\big\{ x^2+y^2=1, z = h \big\}\cup \big\{ x^2+y^2=R^2, z = -h \big\}, \qquad R\ge1,\quad h>0,
\end{equation*}
that is a disjoint union of two parallel circles of possibly different radii. We consider the class of immersions
\begin{equation*}
\sF_{R,h}:=\left\{\vp:\sC\to\R^3\,|\,\vp \mbox{ smooth immersion},\,\,\vp|_{\pa \sC}:\pa \sC\to\Ga_{R,h} \mbox{ smooth embedding}\right\}.
\end{equation*}
By Corollary 3 in \cite{Sc83}, if a minimal surface has $\Ga_{R,h}$ as boundary, then it necessarily is a catenoid or a pair of planar disks. Moreover there exists a threshold value $h_0>0$ such that $\Ga_{R,h}$ is the boundary of a catenoid if and only if $h\le h_0$. For example, in the case of $R=1$ one has $h_0=\left(\min_{t>0} \frac{\cosh (t)}{t}\right)^{-1}$.
In particular for any $h>h_0$ there are no minimal surfaces (and thus no solutions of the Plateau's problem) connecting the two components of $\Ga_{R,h}$, even in a perturbative setting $h\simeq h_0+\ep$. This rigidity in the behavior of minimal surfaces suggests that in some cases an energy different from the Area functional may be a good model for connected soap films, like for describing the optimal elastic surface connecting $\Ga_{R,h}$ in the perturbative case $h\simeq h_0+\ep$. Since surfaces with zero Willmore energy recover critical points of the Plateau's problem, we expect the minimization of $\cW$ to be a good process for describing optimal elastic surfaces under constraints, like connectedness ones, that do not match with the Area functional.\\
Also, from the modeling point of view, we remark the importance of Willmore-type energies, like the Helfrich energy, in the physical study of biological membranes (\cite{ElFrHo}, \cite{SeFr}), and in the theory of elasticity in engineering (see \cite{GaGrSw} and references therein).\\
\noindent We have to mention some remarkable results about critical points of the Willmore energy (called Willmore surfaces) with boundary. Apart from the above cited \cite{ScBP}, Willmore surfaces with a boundary also of the form $\Ga_{R,h}$ have been studied together with the rotational symmetry of the surface in \cite{BeDaFr}, \cite{DaDeGr}, \cite{DaFrGrSc}, \cite{DeGr}, and \cite{Ei}; a new result about symmetry breaking is \cite{Mandel}. Also, interesting results about Willmore surfaces in a free boundary setting is contained in \cite{AlKu}. A relation between Willmore surfaces and minimal surfaces is investigated in \cite{BeJa}.

\medskip

\subsection{Main results} Let us collect here the main results of the paper. If $\ga=\ga^1\cup...\cup\ga^\alpha$ is a disjoint union of smooth embedded compact $1$-dimensional manifolds, we give a sufficient condition guaranteeing existence in minimization problems of the form \eqref{P} or \eqref{Q}. We obtain the following two Existence Theorems.

\begin{reptheorem}{thm:Existence4pinu}
	Let $\ga=\ga^1\cup...\cup\ga^\alpha$ be a disjoint union of smooth embedded compact $1$-dimensional manifolds with $\alpha\in \N_{\ge2}$.\\
	Let
	\[
	\si_0=\nu_0\, m \,\cH^1\res \gamma
	\]
	be a vector valued Radon measure, where $m:\ga\to\N_{\ge1}$ and $\nu_0:\ga\to (T\ga)^\perp$ are $\cH^1$-measurable functions with $m\in L^\infty(\cH^1\res\ga)$ and $|\nu_0|=1$ $\cH^1$-ae.\\
	Let $\cP$ be the minimization problem
	\begin{equation}
	\cP\,:=\, \min\left\{ \,\cW(V) \quad|\quad V=\bv(M,\te_V):\quad
	\si_V=\si_0,\quad \supp V\cup \ga \,\, \mbox{ compact, connected } \right\}.
	\end{equation}
	If $\inf\cP<4\pi$, then $\cP$ has minimizers.
\end{reptheorem}

\begin{reptheorem}{thm:Existence4pi}
	Let $\ga=\ga^1\cup...\cup\ga^\alpha$ be a disjoint union of smooth embedded compact $1$-dimensional manifolds with $\alpha\in \N_{\ge2}$.\\
	Let $m:\ga\to\N_{\ge1}$ by $\cH^1$-measurable with $m\in L^\infty(\cH^1\res\ga)$.\\
	Let $\cQ$ be the minimization problem
	\begin{equation}
	\cQ\,:=\, \min\left\{ \,\cW(V) \quad|\quad V=\bv(M,\te_V):\quad
	|\si_V|\le m\cH^1\res\ga,\quad \supp V\cup \ga \,\, \mbox{ compact, connected } \right\}.
	\end{equation}
	If $\inf\cP<4\pi$, then $\cP$ has minimizers.
\end{reptheorem}

\noindent Both Existence Theorems are obtained by applying a direct method in the context of varifolds. In both cases the connectedness constraint passes to the limit by means of the following theorem, that relates varifolds convergence with convergence in Hausdorff distance of the supports of the varifolds.

\begin{reptheorem}{thm:HausdorffDistance}
	Let $V_n=\bv(M_n,\te_{V_n})\neq0$ be a sequence of curvature varifolds with boundary with uniformly bounded Willmore energy converging to $V=\bv(M,\te_V)\neq0$. Suppose that the $ M_n$'s are connected and uniformly bounded.\\
	Suppose that $\supp\si_{V_n}=\ga^1_n\cup...\cup \ga^\alpha_n$ where the $\ga^i_n$'s are disjoint compact embedded $1$-dimensional manifolds, $\bar{\ga}^1,..., \bar{\ga}^\beta$ with $\beta\le \alpha$ are disjoint compact embedded $1$-dimensional manifolds, and assume that $\ga^i_n\to\bar{\ga}^i$ in $d_\cH$ for $i=1,...,\beta$ and that $\cH^1(\ga^i_n)\to0$ for $i=\beta+1,...,\alpha$.\\
	Then $M_n\to M\cup \bar{\ga}^1\cup...\cup\bar{\ga}^\beta $ in Hausdorff distance $d_\cH$ (up to subsequence) and $M \cup \bar{\ga}^1\cup...\cup\bar{\ga}^\beta$ is connected. Moreover $\ga^i_n\to \{p_i\}$ in $d_\cH$ for any $i=\be+1,...,\alpha$ for some points $\{p_i\}$, each $p_i\in M$, and $\supp \si_V \con  \bar{\ga}^1\cup...\cup\bar{\ga}^\beta\cup \{p_{\be+1},...,\be_\alpha\}$.\\
\end{reptheorem}

\noindent The paper is organized as follows. In Section 2 we recall the monotonicity formula for curvature varifolds with boundary and its consequences on the structure of varifolds with bounded Willmore energy. Such properties are proved in Appendix B. In Section 3 we prove some properties of the Hausdorff distance and we prove Theorem \ref{thm:HausdorffDistance}. Section 4 is devoted to the proof of the Existence Theorems \ref{thm:Existence4pinu} and \ref{thm:Existence4pi}; we also describe remarkable cases in which such theorems apply, such as in the above discussed perturbative setting. Theorem \ref{thm:HausdorffDistance} and the monotonicity formula give us results also about the asymptotic behavior of connected varifolds with suitable boundedness assumptions; more precisely we prove that rescalings of a sequence of varifolds $V_n$ with $\diam(\supp V_n)\to\infty$ converge to a sphere both as varifolds and in Hausdorff distance (Corollary \ref{cor:ConvergenceSphere}). Finally in Section 6 we apply all the previous results to the motivating case of varifolds with boundary conditions on curves of the type of $\Ga_{R,h}$. We prove that for any $R$ and $h$ the minimization problem of type $\cQ$ has minimizers and their rescalings asymptotically approach a sphere (Corollary \ref{cor:Ex/AsympDoubleCirc}). Appendix A recalls the definitions about curvature varifolds with boundary and a useful compactness theorem.

\medskip

\subsection{Notation}

We adopt the following notation.
\begin{itemize}
	\item The symbol $B_r(p)$ denotes the open ball of radius $r$ and center $p$ in $\R^3$.
	\item The symbol $\lgl\cdot,\cdot\rgl$ denotes the Euclidean inner product.
	\item The symbol $\cH^k$ denotes the $k$-dimensional Hausdorff measure in $\R^3$.
	\item The symbol $d_\cH$ denotes the Hausdorff distance.
	\item If $\vp:\Si\to\R^3$ is a smooth immersion of a $2$-dimensional manifold with boundary, then in local coordinates we denote by $\sff_{ij}$ the second fundamental form, by $\vec{H}$ the mean curvature vector, by $g_{ij}$ the metric tensor, by $g^{ij}$ its inverse, by $\mu_\vp$ the volume measure on $\Si$ induced by $\vp$, and by $co_\vp$ the conormal field.
	\item If $v$ is a vector and $M$ is $2$-rectifiable in $\R^3$, the symbol $(v)^\perp$ denotes the projection of $v$ onto $T_pM^\perp$; hence $v^\perp$ is defined $\cH^2$-ae on $M$ and it implicitly depends on the point $p\in M$.
	\item The symbol $V=\bv(M,\te_V)$ denotes an integer rectifiable varifold. Also $\mu_V=\te_V\cH^2\res M$ is the weight measure. If they exist, the generalized mean curvature and boundary are usually denoted by $\vec{H}$ (or $\vec{H}_V$) and $\si_V$.
	\item The symbol $\sC$ denotes a fixed cylinder, i.e. $\sC=[0,1]^2/_\sim$.
	\item For given $R\ge1$ and $h>0$, the symbol $\Ga_{R,h}$ denotes an embedded $1$-dimensional manifold of the form
	\begin{equation*}
	\Ga_{R,h}:=\big\{ x^2+y^2=1, z = h \big\}\cup \big\{ x^2+y^2=R^2, z = -h \big\}, \qquad R\ge1,\quad h>0,
	\end{equation*}
	that is a disjoint union of two parallel circles of possibly different radii. Observe that the distance between the two circles is equal to $2h$.
	\item For a given boundary datum $\Ga_{R,h}$ as above, we define the class
	\begin{equation*}
	\sF_{R,h}:=\left\{\vp:\sC\to\R^3\,|\,\vp \mbox{ smooth immersion},\,\,\vp|_{\pa \sC}:\pa \sC\to\Ga_{R,h} \mbox{ smooth embedding}\right\}.
	\end{equation*}
\end{itemize}


\bigskip

\section{Monotonicity formula and its consequences}

\noindent Here we recall the fundamental monotonicity formula for curvature varifolds with boundary, together with some immediate consequences on surfaces and on the structure of varifolds with finite Willmore energy.\\
This classical formula is completely analogous to its version without boundary (\cite{SiEX}, \cite{KuSc}), hence the technicality behind the results we are going to state is developed in Appendix B.\\

\noindent Let $0<\si<\ro$ and $p_0\in\R^3$. If $V$ is an integer rectifiable curvature varifold with boundary with bounded Willmore energy (here the support of $V$ is not necessarily bounded), with $\mu_V$ the induced measure in $\R^3$, and generalized boundary $\si_V$, it holds that
\begin{equation} \label{monot}
A(\si)+\int_{B_\ro(p_0)\sm B_\si (p_0) } \bigg| \frac{\vec{H}}{2} +\frac{(p-p_0)^\perp}{|p-p_0|^2} \bigg|^2\,d\mu_V(p) = A(\ro),
\end{equation}
where
\begin{equation}
A(\ro):= \frac{\mu_V(B_\ro(p_0))}{\ro^2}+\frac{1}{4}\int_{B_\ro(p_0)} |H|^2\,d\mu_V(p)+ R_{p_0,\ro},
\end{equation}
and
\begin{equation}
\begin{split}
R_{p_0,\ro}&:= \int_{B_\ro(p_0)} \frac{\lgl \vec{H}, p-p_0\rgl}{\ro^2}\,d\mu_V(p) + \frac{1}{2}\int_{B_\ro(p_0)} \bigg( \frac{1}{|p-p_0|^2}-\frac{1}{\ro^2} \bigg)(p-p_0) \,d\si_V(p)\\& =: \int_{B_\ro(p_0)} \frac{\lgl \vec{H}, p-p_0\rgl}{\ro^2}\,d\mu_V(p) + T_{p_0,\ro}.
\end{split}
\end{equation}
In particular the function $\ro\mapsto A(\ro)$ is non-decreasing.

\noindent When more than a varifold is involved, we will usually denote by $A_V(\cdot)$ the monotone quantity associated to $V$ for chosen $p_0\in\R^3$.\\

\noindent It is useful to remember that $T_{p_0,\ro}=0$ if $B_\ro(p_0)\cap \supp\si_V=\epty$, and that

\begin{equation}\label{eq10}
	\left| \int_{B_\ro(p_0)} \frac{\lgl \vec{H}, p-p_0\rgl}{\ro^2}\,d\mu_V(p) \right|\xrightarrow[\ro\to0]{}0
\end{equation}
whenever $\cW(V)<+\infty$ and $p_0\not\in\supp\si_V$ (see \eqref{eq8} in Appendix B).\\

\noindent Let us list some immediate consequences on surfaces with boundary.

\begin{lemma}\label{lem:MonotonicityEstimates}
	Let $\Si\con\R^3$ be a compact connected immersed surface with boundary. Then
	\begin{equation} \label{1}
	\forall p_0\in\R^3:\qquad4\lim_{\si\searrow0} \frac{|\Si\cap B_\si(p_0)|}{\si^2} + 4 \int_{\Si} \bigg| \frac{\vec{H}}{2} +\frac{(p-p_0)^\perp}{|p-p_0|^2} \bigg|^2 = \cW(\Si) + 2 \int_{\pa\Si} \bigg\lgl \frac{p-p_0}{|p-p_0|^2}, co \bigg\rgl.
	\end{equation}
	In particular
	\begin{equation} \label{3}
	\forall p_0\in\R^3\sm\pa\Si:\qquad4\lim_{\si\searrow0} \frac{|\Si\cap B_\si(p_0)|}{\si^2} + 4 \int_{\Si} \bigg| \frac{\vec{H}}{2} +\frac{(p-p_0)^\perp}{|p-p_0|^2} \bigg|^2 \le  \cW(\Si) + 2 \frac{\cH^1(\pa\Si)}{d(p_0,\pa\Si)}.
	\end{equation}
	Moreover calling $d_\cH$ the Hausdorff distance (see Section \ref{SHaus}) and writing $d_\cH (\Si,\pa\Si)= d(\overline{p_0},\pa\Si)$ for some $\overline{p_0}\in\Si\sm\pa\Si$, it holds that
	\begin{equation} \label{2}
	\qquad4\lim_{\si\searrow0} \frac{|\Si\cap B_\si(\overline{p_0})|}{\si^2} + 4 \int_{\Si} \bigg| \frac{\vec{H}}{2} +\frac{(p-\overline{p_0})^\perp}{|p-\overline{p_0}|^2} \bigg|^2 \le  \cW(\Si) + 2 \frac{\cH^1(\pa\Si)}{d_\cH(\Si,\pa\Si)}.
	\end{equation}
\end{lemma}
\begin{proof}
	It suffices to prove \eqref{1}. Since $\Si$ is smooth we have that
	\[
	\left| \int_{B_\ro(p_0)} \bigg( \frac{1}{|p-p_0|^2}-\frac{1}{\ro^2} \bigg)\lgl p-p_0 , co \rgl \,d \cH^1(p) \right| \le \int_{B_\ro(p_0)} \left| \frac{1}{|p-p_0|^2}-\frac{1}{\ro^2} \right| O_{p_0}(|p-p_0|^2) \, d\cH^1(p) \xrightarrow[\ro\to0]{}0.
	\]
	Since $\Si$ is smooth, by \eqref{monot} we have that
	\[
	A(\si)\xrightarrow[\si\to0]{}\lim_{\si\searrow0} \frac{|\Si\cap B_\si(p_0)|}{\si^2},\]
	while by compactness it holds that
	\[
	A(\ro)\xrightarrow[\ro\to\infty]{}\frac14\cW(\Si) + \frac12 \int_{\pa\Si} \big\lgl \frac{p-p_0}{|p-p_0|^2}, co \big\rgl,\]
	and we get \eqref{1}.
\end{proof}

\noindent Let us mention that \eqref{3} already appears in \cite{RiLI}.\\

\noindent More importantly, the monotonicity formula implies fundamental structural properties on varifolds with bounded Willmore energy. First we remark such results in the case of varifolds without boundary, as proved in \cite{KuSc}.

\begin{remark}\label{rem:Rappresentante}
	Let $V=\bv(M,\te_V)$ be an integer rectifiable varifold with $\si_V=0$ and finite Willmore energy. Then at any point $p_0\in\R^3$ there exists the limit
	\begin{equation}\label{eq1}
	\lim_{r\to0} \frac{\mu_V(B_r(p_0))}{\pi r^2}=\te_V(p_0),
	\end{equation}
	and $\te_V$ is upper semicontinuous on $\R^3$ (see (A.7) and (A.9) in \cite{KuSc}). In particular $M=\{p\in\R^3\,\,:\,\,\te_V(p)\ge \frac{1}{2}\}$ is closed.\\
	Recall that if $\supp V$ is also compact and non-empty, then $\cW(V)\ge4\pi$ ((A.19) in \cite{KuSc}) and $\te_V$ is uniformly bounded on $\R^3$ by a constant depending only on $\cW(V)$ ((A.16) in \cite{KuSc}).
\end{remark}

\noindent In complete analogy with Remark \ref{rem:Rappresentante} we prove in Appendix B (see Proposition \ref{prop:StructureProperties}) that if $V$ is a $2$-dimensional integer rectifiable curvature varifold with boundary, denoting by $S$ a compact $1$-dimensional embedded manifold containing the support $\supp\si_V$ with $|\si_V|(S)<+\infty$ and assuming that
	\[
	\cW(V)<+\infty, \quad \limsup_{R\to\infty} \frac{\mu_V(B_R(0))}{R^2}\le K<+\infty,
	\]
then the limit
	\[
	\lim_{\ro\searrow0} \frac{\mu_V(B_\ro(p))}{\ro^2}
	\]
exists at any point $p\in\R^3\sm S$, the multiplicity function $\te_V(p)=\lim_{\ro\searrow0} \frac{\mu_V(B_\ro(p))}{\ro^2}$ is upper semicontinuous on $\R^3\sm S$ and bounded by a constant $C(d(p,S), |\si_V|(S), K,\cW(V))$ depending only on the distance $d(p,S)$, $|\si_V|(S)$, $K$, and $\cW(V)$. Moreover $V=\bv(M,\te_V)$ where $M=\{p\in\R^3\sm S\,|\, \te_v(p)\ge\frac12\}\cup S$ is closed.\\
%

\noindent Whenever a varifold $\bv(M,\te_V)$ satisfies the above assumptions, we will always assume that $M=\{p\in\R^3\sm S\,|\, \te_v(p)\ge\frac12\}\cup S$.\\
These structural properties on curvature varifolds with finite Willmore energy, together with the analogous properties recalled in Remark \ref{rem:Rappresentante}, should be always kept in mind in what follows.

\bigskip

\section{Convergence in the Hausdorff distance}\label{SHaus}

\noindent The convergence of sets with respect to the Hausdorff distance will play an important role in our study. For every sets $X,Y\con\R^3$ we define the Hausdorff distance $d_\cH$ between $X$ and $Y$ by
\begin{equation}
	d_\cH(X,Y):=\inf \left\{\ep>0\,|\,X\con\cN_\ep(Y),\,Y\con\cN_\ep(X)\right\}=\max\left\{\sup_{x\in X} \inf_{y\in Y} |x-y|,\,\sup_{y\in Y} \inf_{x\in X} |x-y|\right\}.
\end{equation}


\noindent We say that a sequence of sets $X_n$ converges to a set $X$ in $d_\cH$ if $\lim_n d_\cH(X_n,X)=0$.\\
Now we prove some useful properties of the Hausdorff distance.

\begin{lemma}\label{lem:HausdorffConnectedness}
	Suppose that $X_n\to X$ in $d_\cH$. Then:\\
	i) $X_n\to\overline{X}$ in $d_\cH$.\\
	ii) If $X_n$ is connected for any sufficiently large $n$ and $X$ is bounded, then $\overline{X}$ is connected as well.
\end{lemma}

\begin{proof}
	i) Just note that if $X\con \cN_{\frac\ep2}(X_n)$, then $\overline{X}\con \cN_\ep(X_n)$.\\
	ii) By $i)$ we can assume without loss of generality that $X$ is closed, and thus compact. Suppose by contradiction that there exist two closed sets $A,B\con X$ such that $A\cap B=\epty$, $A\neq\epty$, $B\neq\epty$, and $A\cup B=X$. Since $X$ is compact, $A$ and $B$ are compact as well, and thus $d(A,B):=\inf_{x\in A,y\in B}|x-y|=\ep>0$. By assumption, for any $n\ge n(\frac\ep4)$ we have that $X_n\con \cN_{\frac\ep4}(X)=\cN_{\frac\ep4}(A)\cup\cN_{\frac\ep4}(B)$ and $\cN_{\frac\ep4}(A)\cap \cN_{\frac\ep4}(B)=\epty$. The sets $\cN_{\frac\ep4}(A)\cap X_n$ and $\cN_{\frac\ep4}(B)\cap X_n$ are disjoint and definitively non-empty, and open in $X_n$. This implies that $X_n$ is not connected for $n$ large enough, that gives a contradiction.
\end{proof}

\begin{lemma}\label{lem:EquivalenceHausdorff}
	Suppose $X_n$ is a sequence of uniformly bounded closed sets in $\R^3$ and let $X\con\R^3$ be closed. Then $X_n\to X$ in $d_\cH$ if and only if the following two properties hold:\\
	a) for any subsequence of points $y_{n_k}\in X_{n_k}$ such that $y_{n_k}\xrightarrow[k]{}y$, we have that $y\in X$,\\
	b) for any $x\in X$ there exists a sequence $y_n\in X_n$ converging to $x$.
\end{lemma}

\begin{proof}
	Suppose first that $d_\cH(X_n,X)\to0$. If there exists a converging subsequence $y_{n_k}\in X_{n_k}$ with limit $y\notin X$, then $d(y_{n_k},X)\ge\ep_0>0$, and thus $X_{n_k}\not\con\cN_{\frac{\ep_0}{2}}(X)$ for $k$ large, that is impossible; so we have proved $a)$. Now let $x\in X$ be fixed. Consider a strictly decreasing sequence $\ep_m\searrow0$ . For any $\ep_m>0$ let $n_{\ep_m}$ be such that $X\con\cN_{\ep_m}(X_n)$ for any $n\ge n_{\ep_m}$. This means that $B_{\ep_m}(x)\cap X_n\neq\epty$ for any $n\ge n_{\ep_m}$ and any $m\in\N$. We can define the sequence
	\[
	n\mapsto x_n \in X_n \cap B_{\ep_{m_n}}(x),
	\]
	where
	\[
	m_n=\sup\left\{m\in \N\,|\,X_n\cap B_{\ep_m}(x)\neq\epty\right\},
	\]
	understanding that $x_n=x$ if $m_n=\infty$, in fact since $X_n$ is closed we have that $x\in X_m$ if $m_n=\infty$. The sequence $\ep_{m_n}$ converges to $0$ as $n\to\infty$, otherwise there exists $\eta>0$ such that $X_n\cap B_\eta(x)=\epty$ for any $n$ large, but this contradicts the convergence in $d_\cH$. Hence $x_n\to x$ and we have proved $b)$.\\
	Suppose now that $a)$ and $b)$ hold. If there is $\ep_0>0$ such that $X_n\not\con\cN_{\ep_0}(X)$ for $n$ large, then a subsequence $x_{n_k}$ converges to a point $y$ such that $d(y,X)\ge\ep_0>0$, that is impossible. If there is $\ep_0>0$ such that $X\not\con\cN_{\ep_0}(X_n)$ for $n$ large, then there is a sequence $z_n\in X$ such that $d(z_n,X_n)\ge\ep_0>0$. By $b)$ we have that $X$ is bounded, then a subsequence $z_{n_k}$ converges to $z\in X$, and $d(z,X_{n_k})\ge\frac{\ep_0}{2}$ definitely in $k$. But then $z$ is not the limit of any sequence $x_{n_k}\in X_{n_k}$. However $z$ is the limit of a sequence $\bar{x}_n\in X_n$ by $b)$, and thus it is the limit of the subsequence $\bar{x}_{n_k}$, and this gives a contradiction.
\end{proof}

\begin{cor}
	Let $X_n$ be a sequence of uniformly bounded closed sets. Suppose that $X_n\to X$ in $d_\cH$ and $X_n\to Y$ in $d_\cH$. If both $X$ and $Y$ are closed, then $X=Y$.
\end{cor}

\begin{proof}
	Both $X$ and $Y$ are bounded. We can apply Lemma \ref{lem:EquivalenceHausdorff}, that immediately implies that $X\con Y$ and $Y\con X$ using the characterization of convergence in $d_\cH$ given by points $a)$ and $b)$.
\end{proof}

\noindent The above properties allow us to relate the convergence in the sense of varifolds to the convergence of their supports in Hausdorff distance. 

\begin{thm}\label{thm:HausdorffDistance}
	Let $V_n=\bv(M_n,\te_{V_n})\neq0$ be a sequence of curvature varifolds with boundary with uniformly bounded Willmore energy converging to $V=\bv(M,\te_V)\neq0$. Suppose that the $ M_n$'s are connected and uniformly bounded.\\
	Suppose that $\supp\si_{V_n}=\ga^1_n\cup...\cup \ga^\alpha_n$ where the $\ga^i_n$'s are disjoint compact embedded $1$-dimensional manifolds, $\bar{\ga}^1,..., \bar{\ga}^\beta$ with $\beta\le \alpha$ are disjoint compact embedded $1$-dimensional manifolds, and assume that $\ga^i_n\to\bar{\ga}^i$ in $d_\cH$ for $i=1,...,\beta$ and that $\cH^1(\ga^i_n)\to0$ for $i=\beta+1,...,\alpha$.\\
	Then $M_n\to M\cup \bar{\ga}^1\cup...\cup\bar{\ga}^\beta $ in Hausdorff distance $d_\cH$ (up to subsequence) and $M \cup \bar{\ga}^1\cup...\cup\bar{\ga}^\beta$ is connected. Moreover $\ga^i_n\to \{p_i\}$ in $d_\cH$ for any $i=\be+1,...,\alpha$ for some points $\{p_i\}$, each $p_i\in M$, and $\supp \si_V \con  \bar{\ga}^1\cup...\cup\bar{\ga}^\beta\cup \{p_{\be+1},...,\be_\alpha\}$.\\
\end{thm}

\begin{proof}
	Let us first observe that by the uniform boundedness of $M_n$, we get that $\ga^i_n$ converges to some compact set $X^i$ in $d_\cH$ up to subsequence for any $i=\be+1,...,\alpha$. Each $X_i$ is connected by Lemma \ref{lem:HausdorffConnectedness}, then by Golab Theorem we know that $\cH^1(X^i)\le \liminf_n \cH^1(\ga^i_n)=0$, hence $X^i=\{p_i\}$ for any $i=\be+1,...,\alpha$ for some points $p_{\be+1},...,p_\alpha$. Call $X=\{p_{\be+1},...,p_\alpha\}$.\\
	By assumption we know that $\mu_{V_n}\wsc\mu_V$ as measures on $\R^3$, also $M_n$ and $M$ can be taken to be closed. Moreover $\supp\si_V\con X\cup \bar{\ga}^1\cup...\cup\bar{\ga}^\beta$. In fact $V_n$ are definitely varifolds without generalized boundary on any open set of the form $\cN_\ep(X\cup \bar{\ga}^1\cup...\cup\bar{\ga}^\beta)$ and they converge as varifolds to $V$ on such an open set with equibounded Willmore energy.\\
	We want to prove that the sets $M_n$ and $ M \cup X\cup \bar{\ga}^1\cup...\cup\bar{\ga}^\beta$ satisfy points $a)$ and $b)$ of Lemma \ref{lem:EquivalenceHausdorff} and that $X\con M$.\\
	
	\noindent Let $x\in M\cup \bar{\ga}^1\cup...\cup\bar{\ga}^\beta\cup X$. If $x\in \bar{\ga}^1\cup...\cup\bar{\ga}^\beta\cup X$, then by assumption and Lemma \ref{lem:EquivalenceHausdorff} there is a sequence of points in $\supp\si_{V_n}$ converging to $x$. So let $x\in M\sm(\bar{\ga}^1\cup...\cup\bar{\ga}^\beta\cup X)$. We know that there exists the limit $\lim_{\ro\searrow0} \frac{\mu_V(B_\ro(x))}{\pi\ro^2}\ge1$, hence we can write that for any $\ro\in(0,\ro_0)$ with $\ro_0<d(x,\supp\si_V)$ we have that $\mu_V(B_\ro(x))\ge \frac\pi2\ro^2$. There exists a sequence $\ro_m\searrow0$ such that $\lim_n \mu_{V_n}(B_{\ro_m}(x))=\mu_V(B_{\ro_m}(x))$ for any $m$. Hence $M_n\cap B_{\ro_m}(x)\neq\epty$ for any $m$ definitely in $n$. Arguing as in Lemma \ref{lem:EquivalenceHausdorff} we find a sequence $x_n\in M_n$ converging to $x$, and thus the property $b)$ of Lemma \ref{lem:EquivalenceHausdorff} is achieved.\\

	\noindent For any $\ep>0$ let $A_\ep:=\cN_\ep(X \cup \bar{\ga}^1\cup...\cup\bar{\ga}^\beta)$. Let us show that for any $\ep>0$ it occurs that $M_n\sm A_\ep$ converges to $\left(M\cup X\cup  \bar{\ga}^1\cup...\cup\bar{\ga}^\beta\right)\sm A_\ep =M\sm A_\ep$ in $d_\cH$, i.e. we want to check property $a)$ of Lemma \ref{lem:EquivalenceHausdorff} for such sets.\\
	Once this convergence is established, we get that $M_n\to M \cup X\cup  \bar{\ga}^1\cup...\cup\bar{\ga}^\beta$ in $d_\cH$ and we can show that the whole thesis follows. In fact we have that for any $\ep>0$ for any $\eta>0$ it holds that
	\[
	M_n\sm  A_\ep\con \cN_{\eta} \left(M\cup X\cup \bar{\ga}^1\cup...\cup\bar{\ga}^\beta\sm  A_\ep\right), \qquad  \left( M \cup X\cup \bar{\ga}^1\cup...\cup\bar{\ga}^\beta\right)\sm A_\ep\con \cN_{\eta} (M_n\sm  A_\ep), 
	\]
	for any $n\ge n_{\ep,\eta}$. In particular
	\[
	M_n = M_n\sm A_\ep \cup A_\ep \con \cN_\eta(M\sm A_\ep) \cup A_\ep \con \cN_{\eta+2\ep} (M\cup X \cup \bar{\ga}^1\cup...\cup \bar{\ga}^\be),
	\]
	\[
	M \cup X \cup \bar{\ga}^1\cup...\cup \bar{\ga}^\be= \left(M \cup X \cup \bar{\ga}^1\cup...\cup \bar{\ga}^\be\right)\sm A_\ep \cup A_\ep \con \cN_\eta(M_n\sm A_\ep) \cup A_\ep \con \cN_{\eta+2\ep} (M_n),
	\]
	for any $n\ge n_{\ep,\eta}$. Setting $\ep=\eta$ we see that for any $\eta>0$ it holds that
	\[
	M_n \con \cN_{3\eta}\left(M\cup X\cup  \bar{\ga}^1\cup...\cup\bar{\ga}^\beta\right), \qquad \left(M\cup X \cup \bar{\ga}^1\cup...\cup\bar{\ga}^\beta\right) \con \cN_{3\eta}(M_n),
	\]
	for any $n\ge n_{2\eta,\eta}$. Hence $M_n\to M \cup X\cup  \bar{\ga}^1\cup...\cup\bar{\ga}^\beta$ in $d_\cH$. Therefore $M \cup X\cup  \bar{\ga}^1\cup...\cup\bar{\ga}^\beta$ is closed and connected. Moreover we get that $X\con M$, in fact for any $p_i\in X$ for any $K\in\N_{\ge1}$ by connectedness of $M_n$ we find some subsequence $y_{n_k}\in M_n \cap \pa B_{\frac1K}(p_i)$ converging to a point $y_K\in M\cap \pa B_{\frac1K}(p_i)$. Since $M$ is closed, passing to the limit $K\to\infty$ we see that $p_i\in M$. In particular $M_n\to M \cup  \bar{\ga}^1\cup...\cup\bar{\ga}^\beta$ in $d_\cH$ and the proof is completed.\\
	
	\noindent So we are left to prove that $M_n\sm A_\ep$ converges to $\left(M\cup X\cup  \bar{\ga}^1\cup...\cup\bar{\ga}^\beta\right)\sm A_\ep = M\sm A_\ep$ in $d_\cH$ for any fixed $\ep>0$. Consider any converging sequence $y_{n_k}\in M_{n_k}\sm A_\ep$. For simplicity, let us denote $y_n$ such sequence. Suppose by contradiction that $y_n\to y$ but $y\not\in M\cup A_\ep$. Since $M$ is closed, there exist $\ze>0$ such that $B_\ze(y)\cap M=\epty$ for $n$ large. Since $M_n$ is connected and $M\neq\epty$ we can write that $\pa B_\ze(y)\cap M_n\neq\epty$ for any $\si\in(\frac\ze4,\frac\ze2)$ for $n$ large enough. Since $y_n\not\in A_\ep$, up to choosing a smaller $\ze$ we can assume that $B_\ze(y)$ does not intersect $\supp\si_{V_n}$ for $n$ large. Fix $N\in\N$ with $N\ge2$ and consider points
	\[
	z_{n,k}\in \pa B_{\left(1+\frac{k}{N}\right)\frac\ze4}(y) \cap M_n \neq\epty,
	\]
	for any $k=1,...,N-1$.\\
	The open balls
	\[
	\left\{B_{\frac{1}{2N}\frac\ze4}(z_{n,k})\right\}_{k=1}^{N-1}
	\]
	are pairwise disjoint. Passing to the limit $\si\searrow0$, setting $\ro=\frac{\ze}{8N}$, and using Young's inequality in Equation \eqref{monot} evaluated on the varifold $V_n$ at the point $p_0=z_{n,k}$ we get that
	\begin{equation}\label{eq9}
	\begin{split}
		\pi &\le  \frac{\mu_{V_n}\left(B_{\frac{\ze}{8N}}(z_{n,k})\right)}{\left(\frac{\ze}{8N}\right)^2} + \frac14 \int_{B_{\frac{\ze}{8N}}(z_{n,k})} |\vec{H}_{V_n}|^2\,d\mu_{V_n} + \frac{1}{\left(\frac{\ze}{8N}\right)^2} \int_{B_{\frac{\ze}{8N}}(z_{n,k})} \lgl \vec{H}_{V_n} , p- z_{n,k} \rgl \, d\mu_{V_n}(p)
		\\ &\le \frac32\,\frac{\mu_{V_n}\left(B_{\frac{\ze}{8N}}(z_{n,k})\right)}{\left(\frac{\ze}{8N}\right)^2} + \frac34 \int_{B_{\frac{\ze}{8N}}(z_{n,k})} |\vec{H}_{V_n}|^2\,d\mu_{V_n},
	\end{split}
	\end{equation}
	for any $n$ large and any $k=1,...,N-1$. Since
	\[
	\limsup_n \mu_{V_n}\left(B_{\frac{\ze}{8N}}(z_{n,k})\right) \le \limsup_n \mu_{V_n}\left(\overline{B_{\frac{\ze}{2}}(y)}\right) \le \mu_V \left(B_{\frac{3}{4}\ze}(y)\right)=0,
	\]
	summing over $k=1,...,N-1$ in \eqref{eq9} and passing to the limit $n\to\infty$ we get that
	\[
	\pi(N-1)\le\limsup_n \frac34 \sum_{k=1}^{N-1} \int_{B_{\frac{\ze}{8N}}(z_{n_k})} |\vec{H}_{V_n}|^2\,d\mu_{V_n} \le \frac34\limsup_n \cW(V_n).
	\]
	Since $N$ can be chosen arbitrarily big from the beginning, we get a contradiction with the uniform bound on the Willmore energy of the $V_n$'s.\\
	\noindent Hence we have proved that $M_n\to M \cup \bar{\ga}^1\cup...\cup\bar{\ga}^\beta$ in $d_\cH$. By Lemma \ref{lem:HausdorffConnectedness} we get that $M \cup \bar{\ga}^1\cup...\cup\bar{\ga}^\beta$ is connected.	
\end{proof}

\begin{remark}\label{rem:Supporto}
	Arguing as in the second part of the proof of Theorem \ref{thm:HausdorffDistance}, we get the following useful statement.\\
	Assuming $V_n=\bv(M_n,\te_{V_n})\neq0$ is a sequence of curvature varifolds with boundary with uniformly bounded Willmore energy converging to $V=\bv(M,\te_V)\neq0$. Suppose that the $M_n$'s are connected and closed and that $M$ is closed. Suppose that $\supp\si_{V_n}$ is as in Theorem \ref{thm:HausdorffDistance}. If a subsequence $y_{n_k}\in M_{n_k}$ converges to $y$, then $y\in M  \cup \bar{\ga}^1\cup...\cup\bar{\ga}^\beta$.\\
	Observe that the supports $M_n, M$ are not necessarily bounded here.
\end{remark}

\begin{remark}
	The connectedness assumption in Theorem \ref{thm:HausdorffDistance} is essential. Consider in fact the following example: let $M_n=\pa B_1(0) \cup \pa B_{\frac1n(0)}$ and $\te_{V_n}(p)=1$ for any $p\in M_n$. Hence the varifolds $\bv(M_n,\te_{V_n})$ converge to $\bv(\pa B_1(0),1)$ as varifolds and they have uniformly bounded energy equal to $8\pi$, but clearly $M_n$ does not converge to $\pa B_1(0)$ in $d_{\cH}$.
\end{remark}

\begin{remark}\label{rem:HausConvergence2}
	The statement of Theorem \ref{thm:HausdorffDistance} also holds if we assume $\supp\si_{V_n}\con \ga^1_n\cup...\cup\ga^\alpha_n$ and $M_n \cup \ga^1_n\cup...\cup\ga^\alpha_n$ connected. In this case, using the notation of the proof of Theorem \ref{thm:HausdorffDistance}, we have that $M_n \cup \ga^1_n\cup...\cup\ga^\alpha_n$ converges to $M \cup X \cup \bar{\ga}^1\cup...\cup\bar{\ga}^\be$ in $d_\cH$ and $M \cup X \cup \bar{\ga}^1\cup...\cup\bar{\ga}^\be$ is connected.
\end{remark}

\bigskip

\section{Perturbative regime: existence in the class of varifolds}

\noindent Now we want to prove the two main Existence Theorems about boundary valued minimization problems on connected varifolds.

\begin{thm}\label{thm:Existence4pinu}
	Let $\ga=\ga^1\cup...\cup\ga^\alpha$ be a disjoint union of smooth embedded compact $1$-dimensional manifolds with $\alpha\in \N_{\ge2}$.\\
	Let
	\[
	\si_0=\nu_0\, m \,\cH^1\res \gamma
	\]
	be a vector valued Radon measure, where $m:\ga\to\N_{\ge1}$ and $\nu_0:\ga\to (T\ga)^\perp$ are $\cH^1$-measurable functions with $m\in L^\infty(\cH^1\res\ga)$ and $|\nu_0|=1$ $\cH^1$-ae.\\
	Let $\cP$ be the minimization problem
	\begin{equation}
	\cP\,:=\, \min\left\{ \,\cW(V) \quad|\quad V=\bv(M,\te_V):\quad
	\si_V=\si_0,\quad \supp V\cup \ga \,\, \mbox{ compact, connected } \right\}.
	\end{equation}
	If $\inf\cP<4\pi$, then $\cP$ has minimizers.
\end{thm}

\begin{proof}
	Let $V_n=\bv(M_n,\te_{V_n})$ be a minimizing sequence for the problem $\cP$. Call $I=\inf \cP<4\pi$, and suppose without loss of generality that $\cW(V_n)<4\pi$ for any $n$. For any $p_0\in M_n\sm \ga$ passing to the limits $\si\to0$ and $\ro\to\infty$ in the monotonicity formula \eqref{monot} we get
	\[
	4\pi \le \cW(V_n) + 2\frac{|\si_0|(\ga)}{d(p_0,\ga)},
	\]
	then
	\[
	\sup_{p_0\in M_n\sm\ga} d(p_0,\ga)\le 2\frac{|\si_0|(\ga)}{4\pi-\cW(V_n)}\le C(\si_0, I).
	\]
	Hence the sequence $M_n$ is uniformly bounded in $\R^3$. Integrating the tangential divergence of the field $X(p)=\chi(p)\,(p)$ where $\chi(p)=1$ for any $p \in B_{R_0}(0)\supset M_n$ for any $n$ we get that
	\[
	2\mu_{V_n}(\R^3) = \int \div_{TM_n} X \, d\mu_{V_n} = - 2\int \lgl H_{V_n}, X \rgl \,d\mu_{V_n} + \int \lgl X, \nu_0\rgl d|\si_0| \le  C(\si_0,I) \mu_{V_n}(\R^3)^{\frac12} + C(\si_0,I),
	\]
	for any $n$, and then $\mu_{V_n}$ is uniformly bounded. By the classical compactness theorem for rectifiable varifolds (\cite{SiGMT}) we have that $V_n\to V=\bv(M,\te_V)$ in the sense of varifolds (up to subsequence), and $M$ is compact.\\
	By an argument analogous to the proof of Theorem \ref{thm:HausdorffDistance} we can show that $V\neq0$. Suppose in fact that $V=0$. Since $\alpha\ge2$ and the curves $\ga^1,...,\ga^\alpha$ are disjoint and embedded, there exist a embedded torus $\phi:S^1\times S^1\to\R^3\sm \ga$ dividing $\R^3$ into two connected components $A_1,A_2$ such that $A_1\supset \ga^1$ and $A_2\supset \ga^2\cup...\cup\ga^\alpha$. Since $M_n$ is connected and uniformly bounded, there is a sequence of points $y_n\in M_n\cap \phi(S^1\times S^1)$ with a converging subsequence $y_{n_k}\to y$. Observe that there is $\De>0$ such that $d(y_n,\ga)\ge\De$. Since $V=0$ we have that $y\not\in \supp V$. Let $N\ge 4$ be a natural number and consider the balls $\left\{B_{\frac{j}{N}\frac\De2}(y)\right\}_{j=1}^N$. Up to subsequence, for $n$ sufficiently large there is $z_{n,j}\in \pa B_{\frac{j}{N}\frac\De2}(y) \cap M_n$. Also the balls
	\[
	\left\{ B_{\frac{\De}{4N}}(z_{n,j}) \right\}_{j=1}^N
	\]
	are pairwise disjoint. As in \eqref{eq9} we get that
	\[
	\pi \le \frac32 \frac{\mu_{V_n}\left( B_{\frac{\De}{4N}}(z_{n,j}) \right)}{\left( \frac{\De}{4N} \right)^2} +\frac34 \int_{B_{\frac{\De}{4N}}(z_{n,j})} |H_{V_n}|^2\,d\mu_{V_n}
	\]
	for any $j=1,...,N$. Since $\limsup_n \mu_{V_n} \left( B_{\frac{\De}{4N}}(z_{n,j}) \right) \le \mu_V( B_{\frac34 \De}(y))=0$, summing over $j=1,...,N$ and passing to the limit in $n$ we get
	\[
	4\pi\le N\pi \le \frac{3}{4}\lim_n \cW(V_n) \le 3\pi,
	\]
	that gives a contradiction. Hence Theorem \ref{thm:HausdorffDistance} implies that $\supp V\cup \ga= M\cup \ga$ is connected. Since $\cW(V)\le I$ by lower semicontinuity, we are left to show that $\si_V=\si_0$.\\
	Since $\ga$ is smooth we can write that
	\begin{equation}\label{eq17}
	|\pi_{(T\ga)^\perp}(p-q_0)|\le C_\ga|p-q_0|^2
	\end{equation}
	as $p\to q_0$ with $p\in \ga$ for some constant $C_\ga$ depending on the curvature of $\ga$. Let $0<\si< s$ with $s=s(\ga)$ such that \eqref{eq17} holds for $p\in \ga\cap B_s(q)$ for any $q\in \ga$. For any $q_0\in \ga$ the monotonicity formula \eqref{monot} at $q_0$ on $V_n$ gives
	\[
	\begin{split}
		\frac{\mu_{V_n}(B_\si(q_0))}{\si^2}&\le -\frac{1}{\si^2}\int_{B_\si(q_0)} \lgl H_{V_n}, p-q_0\rgl \, d\mu_{V_n}(p) - \frac12 \int_{B_\si(q_0)} \left(\frac{1}{|p-q_0|^2}- \frac{1}{\si^2}\right) \lgl p-q_0, \nu_0\rgl \, d|\si_0|(p) + \lim_{\ro\to\infty} A_{V_n}(\ro) \\
		&\le \cW(V_n)^{\frac12} \left( \frac{\mu_{V_n}(B_\si(q_0))}{\si^2}  \right)^{\frac12} + \frac12 \int_{B_\si(q_0)} \frac{C_\ga|p-q_0|^2}{|p-q_0|^2} + \frac{1}{\si}\, d|\si_0|(p) + \pi +\frac12 \int \frac{\lgl p-q_0, \nu_0 \rgl}{|p-p_0|^2}\,d|\si_0|(p)
		\\
		&\le \cW(V_n)^{\frac12} \left( \frac{\mu_{V_n}(B_\si(q_0))}{\si^2}  \right)^{\frac12} + C_\ga |\si_0|(B_\si(q_0)) + \frac{1}{\si}|\si_0|(B_\si(q_0)) + \pi +\frac12 \frac{1}{s} |\si_0|\left(\ga\sm B_\si(q)\right) 
		\\
		&\le C(I) \left( \frac{\mu_{V_n}(B_\si(q_0))}{\si^2}  \right)^{\frac12} + C(\ga,\si_0).
	\end{split}
	\]
	In particular
	\begin{equation}
		\mu_{V_n}(B_\si(q))\le C(I,\ga,\si_0)\si^2
	\end{equation}
	for any $q_0\in \ga$, any $\si\in(0,s)$, and any $n$.\\
	Consider now any $X\in C^0_c(B_r(q_0))$ for fixed $q_0\in \ga$ and $r\in(0,s)$. By varifold convergence we have that
	\begin{equation}\label{eq19}
	\lim_n -2 \int \lgl H_{V_n}, X \rgl \, d\mu_{V_n} + \int \lgl X, \nu_0 \rgl\,d|\si_0| =  -2 \int \lgl H_V, X \rgl \, d\mu_V + \int \lgl X, \nu_V \rgl\,d|\si_V|, 
	\end{equation}
	where we wrote $\si_V=\nu_V\,|\si_V|$. Now let $m\in \N$ be large and consider the cut off function
	\begin{equation}\label{eq20}
	\Lambda_m(p)=\begin{cases}
		1-md(p,\ga) & d(p,\ga)\le \frac1m,\\
		0 & d(p,\ga)>\frac1m.
	\end{cases}
	\end{equation}
	Take now $X=\La_m Y$ for some $Y\in C^0_c(B_r(q_0))$. We have that
	\[
	\begin{split}
		\limsup_{m\to\infty} \lim_n \left| \int \lgl H_{V_n}, X \rgl\,d\mu_{V_n} \right| &=
		\limsup_{m\to\infty} \lim_n \left| \int_{B_r(q_0)\cap \cN_{\frac1m}(\ga)} \La_m\lgl H_{V_n}, Y \rgl\,d\mu_{V_n} \right|	\\
		&\le \|Y\|_\infty \limsup_m \lim_n \cW(V_n)^{\frac12} \mu_{V_n}\left( B_r(q_0)\cap \cN_{\frac1m}(\ga) \right)^{\frac12}.
	\end{split}
	\]
	Moreover, there exists a constant $C(\ga)$ such that $ B_r(q_0)\cap \cN_{\frac1m}(\ga)\con \cup_{i=1}^{C(\ga)m} B_{\frac2m}(q_i)$ for some points $q_i\in \ga$ and at most $C(\ga)m$ balls $\{B_{\frac2m}(q_i)\}_i$. Hence for $\frac2m<s$ we can estimate
	\[
	\mu_{V_n}\left( B_r(q_0)\cap \cN_{\frac1m}(\ga) \right) \le \sum_{i=1}^{C(\ga)m} \mu_{V_n}\left(B_{\frac2m}(q_i)\right) \le C(\ga)m C(I,\ga,\si_0) \frac{4}{m^2}.
	\]
	Therefore
	\begin{equation}\label{eq22}
		\limsup_{m\to\infty} \lim_n \left| \int \lgl H_{V_n}, X \rgl\,d\mu_{V_n} \right| \le
		\|Y\|_\infty \limsup_m C(I,\ga,\si_0) \frac{1}{\sqrt{m}}=0.
	\end{equation}
	Hence setting $X=\La_m Y$ in \eqref{eq19} and letting $m\to\infty$ we obtain
	\[
	\int \lgl Y,\nu_0\rgl \,d|\si_0| =\int \lgl Y, \nu_V \rgl \,d|\si_V|,
	\]
	for any $Y\in C^0_c(B_r(q_0))$. Since $q_0\in\ga$ is arbitrary we conclude that $\si_V=\si_0$, and thus $V$ is a minimizer.	
\end{proof}

\begin{thm}\label{thm:Existence4pi}
	Let $\ga=\ga^1\cup...\cup\ga^\alpha$ be a disjoint union of smooth embedded compact $1$-dimensional manifolds with $\alpha\in \N_{\ge2}$.\\
	Let $m:\ga\to\N_{\ge1}$ by $\cH^1$-measurable with $m\in L^\infty(\cH^1\res\ga)$.\\
	Let $\cQ$ be the minimization problem
	\begin{equation}
	\cQ\,:=\, \min\left\{ \,\cW(V) \quad|\quad V=\bv(M,\te_V):\quad
	|\si_V|\le m\cH^1\res\ga,\quad \supp V\cup \ga \,\, \mbox{ compact, connected } \right\}.
	\end{equation}
	If $\inf\cP<4\pi$, then $\cP$ has minimizers.
\end{thm}

\begin{proof}
	We adopt the same notation used in the proof of Theorem \ref{thm:Existence4pinu}. In this case the generalized boundaries of the minimizing sequence $V_n=\bv(M_n,\te_{V_n})$ are denoted by $\si_{V_n}=\nu_{V_n}|\si_{V_n}|$, and $|\si_{V_n}|\le m\cH^1\res\ga$. The very same strategy used in Theorem \ref{thm:Existence4pinu} shows that $V_n$ converges up to subsequence in the sense of varifolds to a limit $V=\bv(M,\te_V)\neq0$ with $M\cup \ga$ compact and connected by Theorem \ref{thm:HausdorffDistance} and Remark \ref{rem:HausConvergence2}, and $\cW(V)\le \inf \cQ$. Hence, to see that $V$ is a minimizer, we are left to show that $|\si_V|\le m\cH^1\res\ga$. Calling $\mu:=m\cH^1\res\ga$, we find as in Theorem \ref{thm:Existence4pinu} that there exist constants $C=C(\inf\cQ,\ga,\mu)$ and $s=s(\ga)$ such that
	\[
	\mu_{V_n}(B_\si(q))\le C\si^2,
	\]
	for any $q\in \ga$, any $\si\in (0,s)$, and any $n$ large.\\
	For any $X\in C^0_c(B_r(q_0))$ for fixed $q_0\in \ga$ and $r\in(0,s)$ the convergence of the first variation of varifolds reads
	\begin{equation}\label{eq21}
		\lim_n -2 \int \lgl H_{V_n}, X \rgl \, d\mu_{V_n} + \int \lgl X, \nu_{V_n} \rgl\,d|\si_{V_n}| =  -2 \int \lgl H_V, X \rgl \, d\mu_V + \int \lgl X, \nu_V \rgl\,d|\si_V|,
	\end{equation}
	where we wrote $\si_V=\nu_V|\si_V|$. Now we set $X=\La_m Y$ in \eqref{eq21} for $Y\in C^0_c(B_r(q_0))$ and $\La_m$ as in \eqref{eq20}. Estimating as in \eqref{eq22} and taking the limit $m\to\infty$ we obtain
	\[
	\lim_n  \int \lgl Y, \nu_{V_n} \rgl\,d|\si_{V_n}| =   \int \lgl Y, \nu_V \rgl\,d|\si_V|,
	\]
	that is $\si_{V_n}\wsc\si_V$, and thus $|\si_V|(A)\le\liminf_n |\si_{V_n}|(A)\le \mu(A)$ for any open set $A$. Hence $|\si_V|\le \mu$ and $V$ is a minimizer of $\cQ$.	
\end{proof}

\begin{remark}
	Assuming in the above existence theorems that the connected components of the boundary datum are at least two (i.e. $\alpha\ge2$) is technical, but it is also essential in order to obtain a non-trivial minimization problem, i.e. a problem that does not necessarily reduces to a Plateau's one. In fact if we consider a single closed embedded smooth oriented curve $\ga$, Lemma 34.1 in \cite{SiGMT} guarantees the existence of a minimizing integer rectifiable current $T=\tau(M,\te,\xi)$ with compact support and with boundary $\ga$. Hence by Lemma 33.2  in \cite{SiGMT} the integer rectifiable varifold $V=\bv(M,\te)$ is stationary and $\supp\si_V\con\ga$. Then we can take $M=\supp T$, that is compact. Since $\pa T=\ga$ and $T$ is minimizing, the set $M\cup \ga$ is connected and $\cW(V)$ is trivially zero.	
\end{remark}

\noindent The Existence Theorems \ref{thm:Existence4pinu} and \ref{thm:Existence4pi} can be applied in different perturbative regimes, as discussed in the following corollaries and remarks.

\begin{cor}\label{cor:Perturbative1}
	Let $\ga=\ga^1\cup...\cup\ga^\alpha$ be a disjoint union of smooth embedded compact $1$-dimensional manifolds with $\alpha\in \N_{\ge2}$. Suppose that there exists a compact connected surface $\Si\con\R^3$ with boundary $\pa \Si=\ga$. Let $\ep\in(-\ep_0,\ep_0)$ and $f_\ep:\R^3\to\R^3$ be a smooth family of diffeomorphisms with $f_0=id|_{\R^3}$. For any $\ep$ let
	\[
	\si_\ep=co_{f_\ep(\Si)}\cH^1\res(f_\ep(\ga)),
	\]
	where $co_{f_\ep(\Si)}$ is the conormal field of $f_\ep(\Si)$.\\
	If $\cW(\Si)<4\pi$, there exists $\ep_1>0$ such that if $\ep_0<\ep_1$ the minimization problems
	\begin{equation}
	\cP_\ep\,:=\, \min\left\{ \,\cW(V) \quad|\quad V=\bv(M,\te_V):\quad
	\si_V=\si_\ep,\quad \supp V\cup f_\ep(\ga) \,\, \mbox{ compact, connected } \right\},
	\end{equation}
	\begin{equation}
	\cQ_\ep\,:=\, \min\left\{ \,\cW(V) \quad|\quad V=\bv(M,\te_V):\quad
	|\si_V|\le \cH^1\res(f_\ep(\ga)),\quad \supp V\cup f_\ep(\ga) \,\, \mbox{ compact, connected } \right\},
	\end{equation}
	have minimizers for any $\ep\in(-\ep_0,\ep_0)$.
\end{cor}

\begin{cor}\label{cor:Perturbative2}
	Let $\ga=\ga^1\cup...\cup\ga^\alpha$ be a disjoint union of smooth embedded compact $1$-dimensional manifolds with $\alpha\in \N_{\ge2}$. Suppose that there exists a compact connected minimal surface $\Si\con\R^3$ with boundary $\pa \Si=\ga$. Let $\ep\in(-\ep_0,\ep_0)$ and $f_\ep:\R^3\to\R^3$ be a smooth family of diffeomorphisms with $f_0=id|_{\R^3}$. For any $\ep$ let
	\[
	\si_\ep=co_{f_\ep(\Si)}\cH^1\res(f_\ep(\ga)),
	\]
	where $co_{f_\ep(\Si)}$ is the conormal field of $f_\ep(\Si)$.\\
	Then there exists $\ep_1>0$ such that if $\ep_0<\ep_1$ the minimization problems
	\begin{equation}
	\cP_\ep\,:=\, \min\left\{ \,\cW(V) \quad|\quad V=\bv(M,\te_V):\quad
	\si_V=\si_\ep,\quad \supp V\cup f_\ep(\ga) \,\, \mbox{ compact, connected } \right\},
	\end{equation}
	\begin{equation}
	\cQ_\ep\,:=\, \min\left\{ \,\cW(V) \quad|\quad V=\bv(M,\te_V):\quad
	|\si_V|\le \cH^1\res(f_\ep(\ga)),\quad \supp V\cup f_\ep(\ga) \,\, \mbox{ compact, connected } \right\},
	\end{equation}
	have minimizers for any $\ep\in(-\ep_0,\ep_0)$.
\end{cor}

\begin{remark}
	Many examples in which the Existence Theorems \ref{thm:Existence4pinu} and \ref{thm:Existence4pi} and Corollary \ref{cor:Perturbative1} apply are given by defining the following boundary data. We can consider any compact smooth surface $S$ without boundary such that $\cW(S)<8\pi$. Then the monotonicity formula (see also \cite{KuSc} and \cite{LiYau}) implies that $S$ is embedded. We remark that there exist examples of such surfaces having any given genus (\cite{SiEX} and \cite{BaKu}). Considering any suitable plane $\pi$ that intersects $S$ in finitely many disjoint compact embedded curves $\ga^1,...,\ga^\alpha$, we get that one halfspace determined by $\pi$ contains a piece $\Si$ of $S$ with $\cW(\Si)<4\pi$ and $\pa\Si=\ga^1\cup...\cup\ga^\alpha$. Calling $co_\Si$ the conormal field of $\Si$ we get that problems
	\begin{equation*}
	\cP\,:=\, \min\left\{ \,\cW(V) \quad|\quad V=\bv(M,\te_V):\quad
	\si_V=co_\Si\cH^1\res\pa \Si,\quad \supp V\cup \pa\Si \,\, \mbox{ compact, connected } \right\},
	\end{equation*}
	\begin{equation*}
	\cQ\,:=\, \min\left\{ \,\cW(V) \quad|\quad V=\bv(M,\te_V):\quad
	|\si_V|\le \cH^1\res\pa\Si,\quad \supp V\cup \pa\Si \,\, \mbox{ compact, connected } \right\},
	\end{equation*}
	and suitably small perturbations $\cP_\ep$, $\cQ_\ep$ of them have minimizers.
\end{remark}

\begin{remark}
	Suppose that $\ga=\ga^1\cup...\cup\ga^\alpha$ is a disjoint union of compact smooth embedded $1$-dimensional manifolds and that $\ga$ is contained in some sphere $S^2_R(c)$. Up to translation let $c=0$. If there is a point $N\in S^2_R(0)$ such that for any $i$ the image $\pi_N(\ga^i)$ via the stereographic projection $\pi_N:S^2_R(0)\sm\{N\}\to\R^2$ is homotopic to a point in $\R^2\sm \cup_{i=1}^\alpha \pi_N(\ga^i)$, then the problem
	\begin{equation*}
	\cQ\,:=\, \min\left\{ \,\cW(V) \quad|\quad V=\bv(M,\te_V):\quad
	|\si_V|\le \cH^1\res\ga,\quad \supp V\cup \ga \,\, \mbox{ compact, connected } \right\},
	\end{equation*}
	has minimizers. In fact under such assumption there exists a connected submanifold $\Si$ of $S^2_R(0)$ with $\pa\Si=\ga$, thus $\cW(\Si)<4\pi$ and Theorem \ref{thm:Existence4pi} applies.
\end{remark}

\begin{remark}\label{rem:PerturbativeCatenoids}
	For given $R\ge1$ and $h>0$ consider the curves
	\[
	\Ga_{R,h}=\{x^2+y^2=1,z=h\}\cup \{x^2+y^2=R^2,z=-h\}.
	\]
	Suppose that $h_0>0$ is the critical value for which a connected minimal surface $\Si$ with $\pa \Si=\Ga_{R,h}$ exists if and only if $h\le h_0$. Let $\Si_0$ be a minimal surface with $\pa\Si_0=\Ga_{R,h_0}$. Applying Corollary \ref{cor:Perturbative2} we get that for $\ep>0$ sufficiently small the minimization problem
	\[
	\cQ_\ep:=\min\left\{  \,\cW(V) \quad|\quad V=\bv(M,\te_V):\quad
	|\si_V|\le \cH^1\res\Ga_{R,h_0+\ep},\quad \supp V\cup \Ga_{R,h_0+\ep} \,\, \mbox{ compact, connected }  \right\}
	\]
	has minimizers.\\
	Let us anticipate that in the case of boundary data of the form $\Ga_{R,h}$ we will see in Corollary \ref{cor:Ex/AsympDoubleCirc} that actually existence of minimizers for the problem $\cQ_\ep$ is guaranteed for any $\ep>0$.
\end{remark}

\bigskip

\section{Asymptotic regime: limits of rescalings}


\noindent As we recalled in Remark \ref{rem:Rappresentante}, it is proved in \cite{KuSc} that the infimum of the Willmore energy on closed surfaces coincide with the infimum taken over non-zero compact varifolds without boundary. First we prove that such infima are both achieved by spheres. This result is certainly expected by experts in the field, but up to the knowledge of the authors it has not been proved yet without appealing to highly non-trivial regularity theorems.

\begin{prop}\label{prop:Uniqueness}
	Let $V=\bv(M,\te_V)$ be an integer rectifiable varifold with $\si_V=0$ and such that $\supp V$ is compact. If $\cW(V)=4\pi$, then $V=\bv(S^2_R(z),1)$ for some $2$-sphere $S^2_R(z)\con\R^3$.
\end{prop}

\begin{proof}
		Passing to the limits $\si\to0$ and $\ro\to\infty$ in the monotonicity formula for varifolds we get that
	\[
	4\pi\te_V(p_0)+ 4 \int_M \left|\frac{\vec{H}}{2} +\frac{(p-p_0)^\perp}{|p-p_0|^2} \right|^2\,d\mu_V = 4\pi,
	\]
	for any $p_0\in\R^3$. Hence $\te_V(p_0)=1$ for any $p_0\in M$, and also
	\begin{equation}\label{eq2}
	\vec{H}(p)=-2 \frac{(p-p_0)^\perp}{|p-p_0|^2},
	\end{equation}
	for $\cH^2$-ae $p\in M$ and for every $p_0\in M$.\\
	Fix $\de>0$ small and two points $p_1,p_2\in M$ with $p_2\not\in B_{2\de}(p_1)$. For $\cH^2$-ae $p\in M$ we can write
	\[
	\vec{H}(p)=\begin{cases}
	-2 \frac{(p-p_1)^\perp}{|p-p_1|^2} & p\not\in B_\de(p_1),\\
	-2 \frac{(p-p_2)^\perp}{|p-p_2|^2} &p\not\in B_\de(p_2).
	\end{cases}
	\]
	Since $M$ is bounded, we get that $\vec{H}\in L^\infty(\mu_V)$. Therefore, since $\te_V=1$ on $M$, by the Allard Regularity Theorem (\cite{SiGMT}) we get that $M$ is a closed surface of class $C^{1,\alpha}$ for any $\alpha\in(0,1)$.\\
	Since $M$ is closed, it is also compact, and thus it is connected, for otherwise $\cW(V)\ge8\pi$.\\
	\noindent Let $p\in M$ be any fixed point such that \eqref{eq2} holds, and call $\nu_p$ the unit vector such that $\nu_p^\perp=T_pM$. Up to translation let $p=0$. Consider the axis generated by $\nu_0$ and any point $p_0\in M\sm\{0\}$. We can write $p_0=q+w$ with $q=\alpha\nu_0$ and $\lgl w, \nu_0 \rgl =0$. Writing analogously $(q+w')\in M\sm\{0\}$ another point with the same component on the axis generated by $\nu_0$, \eqref{eq2} implies that
	\[
	-2\frac{-\lgl q, \nu_0\rgl \nu_0}{|q|^2+|w|^2}=-2\frac{(0-q-w)^{\perp_0}}{|q-w|^2}=\vec{H}(0)=-2\frac{(0-q-w')^{\perp_0}}{|q-w'|^2}=-2\frac{-\lgl q, \nu_0\rgl \nu_0}{|q|^2+|w'|^2}.
	\]
	Hence, whenever $q\ne0$, we have that $|w|=|w'|$; that is points in $M$ of the form $\alpha\nu_0 + w$ with $\alpha\neq0$ and $w\in\nu_0^\perp$ lie on a circle. It follows that $M$ is invariant under rotations about the axis $\{t\nu_0\,\,|\,\,t\in\R\}$.\\
	This argument works at $\cH^2$-almost any point of $M$. Therefore we have that for any $p\in M$, the set $M$ is invariant under rotations about the axis $p+\{t\nu_p\,\,|\,\,t\in\R\}$.\\
	Still assuming $0\in M$, up to rotation suppose that $\nu_0=(0,0,1)$. Let $a\in M$ be such that $\nu_a=(1,0,0)$. There exists a point $b\in M$ such that $b=t\nu_0=(0,0,t)$ for some $t\in\R\sm\{0\}$. We can write $0=q+w$ and $b=q+w'$ for the same $q\in a+\{t\nu_a\,\,|\,\,t\in\R\}$ and some $w,w'\in\nu_a^\perp$. Since $|w|=|w'|$, it follows that $q\neq0$, otherwise $b=0$. Since $q\neq0$, the rotation of the origin about the axis $a+\{t\nu_a\,\,|\,\,t\in\R\}$ implies that $M$ contains a circle $C$ of radius $r>0$ passing through the origin, and the plane containing $C$ is orthogonal to $\nu_0^\perp$. Since $M$ is of class $C^1$, the circle $C$ has to be tangent at $0$ to the subspace $\nu_0^\perp$. Thus by invariance with respect to the rotation about the axis $\{t\nu_0\,|\,t\in\R\}$, we have that $M$ contains the sphere with positive radius given by the rotation of $C$ about $\{t\nu_0\,|\,t\in\R\}$. Since the Willmore energy of a sphere is $4\pi$, it follows that $M$ coincide with such sphere.
\end{proof}

\noindent Now we can prove the above mentioned result on the asymptotic behavior of connected varifolds.

\begin{cor}\label{cor:ConvergenceSphere}
	Let $V_n=\bv(M_n,\te_{V_n})$ be a sequence of integer rectifiable curvature varifolds with boundary satisfying the hypotheses of Theorem \ref{thm:ConvergenceVarifolds}. Suppose that $M_n$ is compact and connected for any $n$.\\
	If
	\[
	\begin{split}
	&\cW(V_n)\le 4\pi +o(1) \qquad\mbox{as $n\to\infty$},\\
	& \diam (\supp V_n) \xrightarrow[n\to\infty]{}+\infty,\\
	&\limsup_n \frac{|\si_{V_n}|(\R^3)}{\diam (\supp V_n)}=0,
	\end{split}
	\]
	and $\supp\si_{V_n}$ is a disjoint union of uniformly finitely many compact embedded $1$-dimensional manifolds, then the sequence
	\[
	\tV_n:=\bv\left( \frac{M_n}{\diam  (\supp V_n) }, \tilde{\te}_n \right)
	\]
	where $\tilde{\te}_n(x)=\te_{V_n}(\diam  (\supp V_n) \, x)$, converges up to subsequence and translation to the varifold
	\[
	V=\bv(\S, 1),
	\]
	where $\S$ is a sphere of diameter $1$, in the sense of varifolds and in Hausdorff distance.
\end{cor}

\begin{proof}
	Up to translation let us assume that $0\in \supp V_n$. Then $\supp\tV_n$ is uniformly bounded with $\diam(\supp\tV_n)=1$. We have that
	\[
	2\mu_{\tV_n}(\R^3)= \int \div_{T\tV_n} p\,d\mu_{\tV_n}(p) \le C \cW(\tV_n)^{\frac12}\left( \mu_{\tV_n}(\R^3)\right)^{\frac12} + C \frac{|\si_{V_n}|(\R^3)}{\diam (\supp V_n)},
	\]
	and thus Theorem \ref{thm:ConvergenceVarifolds} implies that $\tV_n$ converges to a limit varifold $V$ (up to subsequence). Also $\si_{\tV_n}\wsc\si_V$, and thus $|\si_V|(\R^3)\le \liminf_n |\si_{\tV_n}|(\R^3)\le \limsup_n \frac{|\si_{V_n}|(\R^3)}{\diam (\supp V_n)}=0$; hence $V$ has compact support and no generalized boundary.\\
	Let us say that $\supp \si_{\tV_n}$ is the disjoint union of the smooth closed curves $\ga^1_n,...,\ga^\alpha_n$. By the uniform boundedness of $\supp\tV_n$, we get that $\ga^i_n$ converges to some compact set $X^i$ in $d_\cH$ up to subsequence. Each $X_i$ is connected by Lemma \ref{lem:HausdorffConnectedness}, then by Golab Theorem we know that $\cH^1(X^i)\le \liminf_n \cH^1(\ga^i_n)=0$, hence $X^i=\{p_i\}$ for any $i$ for some points $p_1,...,p_\alpha$, and we can assume that $p_i\neq0$ for any $i=1,...,\alpha$.\\
	Using ideas from the proof of Theorem \ref{thm:HausdorffDistance}, we can show that $V\neq 0$.	In fact suppose by contradiction that $V=0$. Fix $N\in \N$ with $N\ge4$. By connectedness of $M_n$, since $\diam(\supp\tV_n)\to1$, and the boundary curves converge to a discrete sets, for $j=1,...,N$ there are points $z_{n,j}\in \pa B_{\frac{j}{2N}} (0)\cap \supp\tV_n$ for $n$ large. We can also choose $N$ so that $d(z_{n,j},\supp\si_{\tV_n})\ge\de(N)>0$ for $n$ large. The open balls $\left\{ B_{\frac{1}{4N}}(z_{n,j}) \right\}_{j=1}^N$ are pairwise disjoint. Using Young inequality as in Theorem \ref{thm:HausdorffDistance} in the monotonicity formula \eqref{monot} applied on $\tV_n$ at points $z_{n,j}$ with $\si\to0$ and $\ro=\frac{1}{4N}$ gives
	\begin{equation}\label{eq18}
	\pi \le \frac32\frac{\mu_{\tV_n}(B_{\frac{1}{4N}}(z_{n,j}))}{\left(\frac{1}{4N}\right)^2} + \frac34 \int_{B_{\frac{1}{4N}}(z_{n,j})} |H_{\tV_n}|^2\,d\mu_{\tV_n}+\frac12 \left| \int \left( \frac{1}{|p-z_{n,j}|^2}-\frac{1}{\left(\frac{1}{4N}\right)^2} \right)(p-z_{n,j})\,d\si_{\tV_n}(p)\right|,
	\end{equation}
	for any $n$ and $j=1,...,N$. Since $V=0$ we have that $\limsup_n \mu_{\tV_n}(B_{\frac{1}{4N}}(z_{n,j})) \le \limsup_n \mu_{\tV_n}(\overline{B_2(0)})=0 $. Also
	\[
	\left| \int \left( \frac{1}{|p-z_{n,j}|^2}-\frac{1}{\left(\frac{1}{4N}\right)^2} \right)(p-z_{n,j})\,d\si_{\tV_n}(p) \right|\le C(\de(N),N) |\si_{\tV_n}|(\R^3)\xrightarrow[n\to\infty]{}0.
	\]
	Hence summing on $j=1,...,N$ in \eqref{eq18} and passing to the limit $n\to\infty$ we get
	\[
	4\pi\le N\pi \le \frac34\lim_n \cW(\tV_n)\le 3\pi,
	\]
	that gives a contradiction.\\
	Therefore we can apply Theorem \ref{thm:HausdorffDistance} to conclude that $\supp\tV_n$ converges to $M$ in $d_\cH$. Finally, since $V$ is a compact varifold without generalized boundary and
	\[
	4\pi\le \cW(V)\le\liminf_n \cW(V_n)=4\pi,
	\]
	by Proposition \ref{prop:Uniqueness} we conclude that $V$ is a round sphere of multiplicity $1$. By Lemma \ref{lem:EquivalenceHausdorff} the diameter of $M$ is the limit $\lim_n \diam(\supp\tV_n)=1$.	
\end{proof}

%

\bigskip

\section{The double circle boundary}

\noindent In this section we want to discuss how the Existence Theorems \ref{thm:Existence4pinu} and \ref{thm:Existence4pi} and the asymptotic behavior described in Corollary \ref{cor:ConvergenceSphere} relate with the remarkable case that motivates our study, namely the immersions in the class $\sF_{R,h}$.\\

\noindent First, the monotonicity formula provides the following estimates on immersions $\vp\in\sF_{R,h}$.

\begin{lemma}\label{lem:StimeOvvie}
	Fix $R\ge1$ and $h>0$. It holds that:\\
	i)
	\begin{equation}
	\inf \big\{ \cW(\vp)\,\,|\,\, \vp\in \sF_{R,h}\big\}\le 4\pi \frac{4h^2+R^2-1}{\sqrt{(4h^2+R^2-1)^2+16h^2}}<4\pi.
	\end{equation}
	ii)
	\begin{equation}
	\lim_{h\to\infty} \inf \big\{  \cW(\vp)\,\,|\,\, \vp\in \sF_{R,h}\big\} = 4\pi.
	\end{equation}
\end{lemma}

\begin{proof}
	$i)$ We can consider as competitor in $\sF_{R,h}$ the truncated sphere
	\[
	\Si=S^2_{\sqrt{1+(z_0-h)^2}}\left(z_0 \right) \cap \left\{|z|\le h\right\},
	\]
	where $z_0=\left(0,0,\frac{1-R^2}{4h}\right)$ is the point on the $z$-axis located at the same distance from the two connected components of $\Ga_{R,h}$. The surface $\Si$ is contained in another truncated sphere $\Si'$ having the same center and radius and symmetric with respect to the plane $\{z=\frac{1-R^2}{4h}\}$. The boundary of $\Si'$ is the disjoint union of two circles of radius $1$. We have
	\[
	\cW(\Si)\le\cW(\Si')=4\pi \frac{4h^2+R^2-1}{\sqrt{(4h^2+R^2-1)^2+16h^2}}
	\]
	$ii)$ Let $\vp\in\sF_{R,h}$ and $\Si=\vp(\sC)$. By connectedness there is a point $p\in\Si\sm\pa\Si$ lying in the plane $z=0$. Hence $d_\cH(\Si,\pa\Si)\ge h$, and by \eqref{2} we have
	\begin{equation*}
	4\pi\le \cW(\Si) + 2 \frac{2\pi(1+R)}{h} \qquad\forall\Si.
	\end{equation*}
	Then $4\pi\le \inf \big\{  \cW(\vp)\,\,|\,\, \vp\in \sF_{R,h}\big\} + \frac{4\pi(1+R)}{h}$ and the thesis follows by using $i)$ by letting $h\to\infty$.\\
\end{proof}

\noindent We already discussed in Remark \ref{rem:PerturbativeCatenoids} the existence of minimization problems arising by perturbations of minimal catenoids in some $\sF_{R,h}$. By Lemma \ref{lem:StimeOvvie} we can complete the picture about existence of optimal connected elastic surfaces with boundary $\Ga_{R,h}$ for any $R\ge1$ and $h>0$, as well as the asymptotic behavior of almost optimal surfaces having such boundaries.

\begin{cor}\label{cor:Ex/AsympDoubleCirc}
	Fix $R\ge1$ and $h>0$.\\
	1) Then the minimization problem
	\[
	\cQ_{R,h}:= \min\left\{  \,\cW(V) \quad|\quad V=\bv(M,\te_V):\quad
	|\si_V|\le \cH^1\res\Ga_{R,h},\quad \supp V\cup \Ga_{R,h} \,\, \mbox{ compact, connected }  \right\}
	\]
	has minimizers.\\
	2) Let $h_k\to\infty$ be any sequence. Let $\Si_k=\vp_k(\sC)$ for $\vp_k\in\sF_{R,h_k}$. Suppose that $\cW(\vp_k)\le 4\pi+o(1)$ as $k\to\infty$. Let $S_k=\frac{\Si_k}{\diam\Si_k}$.\\
	Then (up to subsequence) $S_k$ converges in Hausdorff distance to a sphere $\S$ of diameter $1$, and the varifolds corresponding to $S_k$ converge to $V=\bv(\S,1)$ in the sense of varifolds.
\end{cor}

\begin{proof}
	1) The result follows by point $i)$ in Lemma \ref{lem:StimeOvvie} by applying Corollary \ref{cor:Perturbative1}.\\
	2) Identifying $S_k$ with the varifold it defines, we estimate the total variation of the boundary measure by $|\pa S_k|\le \frac{\cH^1(\Ga_{R,h_k})}{\diam\Si_k}$. Moreover, by the Gauss-Bonnet Theorem the $L^2$-norm of the second fundamental form of $S_k$ is uniformly bounded. Hence Corollary \ref{cor:ConvergenceSphere} applies and the thesis follows.
\end{proof}

\noindent Using the notation of point $2)$ in Corollary \ref{cor:Ex/AsympDoubleCirc}, we remark that even if we know that the rescalings $S_k$ converge to a sphere in $d_\cH$ and as varifolds, it remains open the question whether at a scale of order $h$ the sequence $\Si_k$ approximate a big sphere. More precisely it seems a delicate issue to understand if $\diam\Si_k \sim 2h_k$ as $k\to\infty$.\\
We conclude with the following partial result: the monotonicity formula gives us some evidence in the case we assume that $\frac{\diam \Si_k}{h_k}\to\infty$.

\begin{prop}\label{prop:ConvergenzaPiani}
	Let $\Si_k=\vp_k(\sC)$ for $\vp_k\in\sF_{R,h_k}$. Suppose that $\cW(\vp_k)\le 4\pi+o(1)$ as $k\to\infty$. Let $M_k=\frac{\Si_k}{h_k}$.\\
	Then $M_k$ converges up to subsequence to $Z=\bv(M,\te_Z)$ in the sense of varifolds.\\
	If also
	\[
	\frac{\diam \Si_k}{h_k}\to\infty,
	\]
	then $M$ is a plane containing the $z$-axis and $\te_Z\equiv1$.
\end{prop}

\begin{proof}
	We identify $M_k$ with the varifold it defines. First we can establish the convergence up to subsequence in the sense of varifolds by using Theorem \ref{thm:ConvergenceVarifolds}. In fact we have that $\cH^1(\pa M_k)\to0$, $\int_{M_k} |\sff_{M_k}|^2$ is scaling invariant and thus finite. 
	Moreover, since $d(0,\pa M_k)\ge1$, by monotonicity \eqref{monot} we get that
	\[
	\begin{split}
		\frac{\mu_{M_k}(B_\si(0))}{\si^2} &\le -\frac{1}{\si^2}\int_{B_\si(0)} \lgl H_{M_k}, p\rgl\,d\mu_{M_k}(p)-\frac12 \int_{B_\si(0)\cap \pa M_k} \left(\frac{1}{|p|^2}-\frac{1}{\si^2}\right)\lgl p,co_{M_k}(p)\rgl \, d\cH^1(p)  \\&\qquad+\lim_{\ro\to\infty} A_{M_k}(\ro) \\
		&\le \pi+o(1) +\frac{1}{\si^2}\int_{B_\si(0)} |p||H_{M_k}|\,d\mu_{M_k}(p) +\frac12 \int_{\pa M_k\sm B_\si(0)} \frac{d\cH^1(p)}{|p|} \\&\qquad+\frac{1}{2\si^2}\int_{\pa M_k \cap B_\si(0)} |p|\,d\cH^1(p)\\
		&\le \pi+o(1) +\frac1\si \mu_{M_k}(B_\si(0))^{\frac12} \cW(M_k)^{\frac12} +\frac12 \cH^1(\pa M_k) + \frac{1}{2\si}\cH^1(\pa M_k),
	\end{split}
	\]
	where $A_{M_k}(\cdot)$ is the monotone quantity centered at $0$ evaluated on $M_k$, and therefore $\mu_{M_k}(B_\si(0))\le C(\si)$ for any $\si\ge1$. Hence the hypotheses of Theorem \ref{thm:ConvergenceVarifolds} are satisfied and we call $Z=\bv(M,\te_Z)$ the limit varifold of $M_k$. Observe that $\si_Z=0$ and $\cW(Z)<+\infty$.\\
	From now on assume that $\diam\Si_k/h_k\to\infty$. Arguing as in the proof of Corollary \ref{cor:ConvergenceSphere} we can prove that $Z\neq 0$. In fact suppose by contradiction that $Z=0$. Fix $N\in \N$ with $N\ge4$. By connectedness of $M_k$, for $j=1,...,N$ there are points $z_{k,j}\in \pa B_{\frac{j}{N}} (0,0,1)\cap M_k $ and $z_{k,j}\not\in \pa M_k$ for $k$ large. The open balls $\left\{ B_{\frac{1}{2N}}(z_{k,j}) \right\}_{j=1}^N$ are pairwise disjoint. Hence the monotonicity formula \eqref{monot} applied on $M_k$ at points $z_{k,j}$ with $\si\to0$ and $\ro=\frac{1}{2N}$ gives
	\begin{equation}\label{eq16}
	\pi \le \frac32\frac{\mu_{M_k}(B_{\frac{1}{2N}}(z_{k,j}))}{\left(\frac{1}{2N}\right)^2} + \frac34 \int_{B_{\frac{1}{2N}}(z_{k,j})} |H_{M_k}|^2\,d\mu_{M_k},
	\end{equation}
	for any $k$ and $j=1,...,N$. Since $Z=0$ we have that $$\limsup_k \mu_{M_k}(B_{\frac{1}{2N}}(z_{k,j})) \le \limsup_k \mu_{M_k}(B_{2}(0,0,1))=0.$$ 
	Hence, summing on $j=1,...,N$ in \eqref{eq16} and passing to the limit $k\to\infty$ we get
	\[
	4\pi\le N\pi \le \frac34\lim_k \cW(M_k)\le 3\pi,
	\]
	that gives a contradiction.\\
	Also the support of $Z$ is unbounded. In fact suppose by contradiction that $\supp Z\con\con B_R(0)$, and thus $M$ is closed by Proposition \ref{prop:StructureProperties}. Since $M_k$ is connected, there exists $q_k'\in M_k\cap \pa B_{2R}(0)$ definitely in $k$ for $R$ sufficiently big. Up to subsequence $q'_k\to q'$. By Remark \ref{rem:Supporto} we get that $q'\in \supp Z$, that contradicts the absurd hypothesis.\\
	Since $M$ is unbounded, by Corollary \ref{cor:StimaUnboundedSupport} (or equivalently (A.22) in \cite{KuSc}) we know that
	\[
	\lim_{\ro\to\infty} \frac{\mu_Z(B_\ro(q))}{\ro^2}\ge \pi.
	\]
	By construction
	\[
	\lim_k \int_{B_\si(0)\cap \pa M_k} \left\lgl \frac{p}{|p|^2}, co_{M_k} \right\rgl \, d\cH^1(p) =0,
	\]
	hence passing to the limit $k\to\infty$ in the monotonicity formula \eqref{monot} evaluated on $M_k$ we get that
	\[
	A_Z(\si) \le \liminf_k A_{M_k}(\si),
	\]
	for ae $\si>0$. By monotonicity
	\[
	\begin{split}
		A_Z(\si) \le  \liminf_k \lim_{\si\to\infty} A_{M_k}(\si) \le  \liminf_k \frac{\cW(M_k)}{4} + \cH^1(\pa M_k) \le \pi.
	\end{split}
	\]
	On the other hand, by (A.14) in \cite{KuSc} we can write that
	\[
	\lim_{\si\to\infty} A_Z(\si) = \frac14 \cW(Z)+ \lim_{\si\to\infty} \frac{\mu_Z(B_\si(q))}{\si^2} \ge \frac14 \cW(Z)+ \pi.
	\]
	Hence $Z$ is stationary, $\lim_{\ro\to\infty} \frac{\mu_Z(B_\ro(q))}{\ro^2}=\pi$, and $M$ is closed.\\
	If $p_0$ is any point in $M$, the monotonicity formula for $Z$ centered at $p_0$ reads
	\begin{equation}\label{eq14}
	\frac{\mu_Z(B_\si(p_0))}{\si^2} + \int_{B_\ro(p_0)\sm B_\si(p_0)} \frac{|(p-p_0)^\perp|^2}{|p-p_0|^4}= \frac{\mu_Z(B_\ro(q))}{\ro^2}.
	\end{equation}
	In particular $\te_Z(p_0)=1$, and thus we can apply Allard Regularity Theorem at $p_0$. Thus we get that $M$ is of class $C^\infty$ around $p_0$ (and analogously everywhere), and thus there exists the limit
	$$\lim_{\si\to0} \int_{B_\ro(p_0)\sm B_\si(p_0)} \frac{|(p-p_0)^\perp|^2}{|p-p_0|^4}= \int_{B_\ro(p_0)} \frac{|(p-p_0)^\perp|^2}{|p-p_0|^4}.$$ 
	Passing to the limits $\ro\to\infty$ and $\si\searrow0$ in \eqref{eq14}, we get that
	\[
	\lim_{\ro\to\infty} \int_{B_\ro(p_0)} \frac{|(p-p_0)^\perp|^2}{|p-p_0|^4}=0.
	\]
	Therefore $|(p-p_0)^\perp|=0$ for any $p\in M$, where we recall that $(\cdot)^\perp$ is the orthogonal projection on $T_pM^\perp$. Since this is true for any $p_0\in M$, we derive that $M$ is a plane.
	Finally Remark \ref{rem:Supporto} implies that $M$ contains the vertical axis $\{(0,0,t)\,|\,t\in\R\}$.
\end{proof}

\bigskip

\appendix

\section{Curvature varifolds with boundary}

\noindent In this appendix we recall the definitions and the results about curvature varifolds with boundary that we need throughout the whole work. This section is based on \cite{Ma} (see also \cite{SiGMT}, \cite{Hu}).\\

\noindent Let $\Om\con\R^k$ be an open set, and let $1<n\le k$. We identify a $n$-dimensional vector subspace $P$ of $\R^k$ with the $k\times k$-matrix $\{P_{ij}\}$ associated to the orthogonal projection over the subspace $P$. Hence the Grassmannian $G_{n,k}$ of $n$-spaces in $\R^k$ is endowed with the Frobenius metric of the corresponding projection matrices. Moreover given a subset $A\con\R^k$, we define $G_n(A)=A\times G_{n,k}$, endowed with the product topology. A general $n$-varifold $V$ in an open set $\Om\con\R^k$ is a non-negative Radon measure on $G_n(\Om)$. The varifold convergence is the weak* convergence of Radon measures on $G_n(\Om)$, defined by duality with $C^0_c(G_n(\Om))$ functions.\\
We denote by $\pi:G_n(\Om)\to\Om$ the natural projection, and by $\mu_V=\pi_\sharp(V)$ the push forward of a varifold $V$ onto $\Om$. The measure $\mu_V$ is called induced (weight) measure in $\Om$.\\
Given a couple $(M,\te)$ where $M\con\Om$ is countably $n$-rectifiable and $\te:M\to\N_{\ge1}$ is $\cH^n$-measurable, the symbol $\bv(M,\te)$ defines the (integer) rectifiable varifold given by
\begin{equation}
\int_{G_n(\Om)} \vp(x,P)\,d\bv(M,\te)(x,P) = \int_M \vp(x,T_xM)\,\te(x)\,d\cH^n(x),
\end{equation}
where $T_xM$ is the generalized tangent space of $M$ at $x$ (which exists $\cH^n$-ae since $M$ is rectifiable). The function $\te$ is called density or multiplicity of $\bv(M,\te)$. Note that $\mu_V=\te\cH^n\res M$ in such a case.\\

\noindent From now on we will always understand that a varifold $V$ is an integer rectifiable one.\\

\noindent We say that a function $\vec{H}\in L^1_{loc}(\mu_V;\R^k)$ is the generalized mean curvature of $V=\bv(M,\te)$ and $\si_V$ Radon $\R^k$-valued measure on $\Om$ is its generalized boundary if
\begin{equation}
\int \div_{TM} X \, d\mu_V = - n \int \lgl \vec{H}, X \rgl \,d\mu_V + \int X\,d\si_V,
\end{equation}
for any $X\in C^1_c(\Om;\R^k)$, where $\div_{TM} X(p)$ is the $\cH^n$-ae defined tangential divergence of $X$ on the tangent space of $M$. Recall that $\si_V$ has the form $\si_V=\nu_V\si$, where $|\nu_V|=1$ $\si$-ae and $\si$ is singular with respect to $\mu_V$.\\

\noindent If $V$ has generalized mean curvature $\vec{H}$, the Willmore energy of $V$ is defined to be
\begin{equation}
\cW(V)=\int |H|^2\,d\mu_V.
\end{equation}
The operator $X\mapsto\de V(X):=\int\div_{TM} X\,d\mu_V$ is called first variation of $V$. Observe that for any $X\in C^1_c(\Om;\R^k)$, the function $\vp(x,P):=\div_{P}(X)(x)=tr(P\nabla X(x))$ is continuous on $G_n(\Om)$. Hence, if $V_n\to V$ in the sense of varifolds, then $\de V_n(X)\to\de V(X)$.\\

\noindent By analogy with integration formulas classically known in the context of submanifolds, we say that a varifold $V=\bv(M,\te)$ is a curvature $n$-varifold with boundary in $\Om$ if there exist functions $A_{ijk}\in L^1_{loc}(V)$ and a Radon $\R^k$-valued measure $\pa V$ on $G_n(\Om)$ such that
\begin{equation}
\begin{split}
	\int_{G_n(\Om)} P_{ij}\pa_{x_j}\vp(x,P) &+ A_{ijk}(x,P)\pa_{P_{jk}}\vp(x,P)   \,dV(x,P) =\\&=  n\int_{G_n(\Om)}\vp(x,P) A_{jij}(x,P) \,dV(x,P) + \int_{G_n(\Om)} \vp(x,P)\,d\pa V_i(x,P),
\end{split}
\end{equation}
for any $i=1,...,k$ for any $\vp\in C^1_c(G_n(\Om))$. The rough idea is that the term on the left is the integral of a tangential divergence, while on the right we have integration against a mean curvature plus a boundary term. The measure $\pa V$ is called boundary measure of $V$.

\begin{thm}[\cite{Ma}]
	Let $V=\bv(M,\te)$ be a curvature varifold with boundary on $\Om$. Then the following hold true.\\
	i) $A_{ijk}=A_{ikj}$, $A_{ijj}=0$, and $A_{ijk}=P_{jr}A_{irk}+P_{rk}A_{ijr}=P_{jr}A_{ikr}+P_{kr}A_{ijr}$.\\
	ii) $P_{il}\pa V_l(x,P)=\pa V_i(x,P)$ as measures on $G_n(\Om)$.\\
	iii) $P_{il}A_{ljk}=A_{ijk}$.\\
	iv) $H_i(x,P):=\frac{1}{n}A_{jij}(x,P)$ satisfies that $P_{il}H_l(x,P)=0$ for $V$-ae $(x,P)\in G_n(\Om)$.\\
	v) $V$ has generalized mean curvature $\vec{H}$ with components $H_i(x,T_xM)$ and generalized boundary $\si_V=\pi_\sharp(\pa V)$.
\end{thm}

\noindent We call the functions $\sff_{ij}^k(x):=P_{il}A_{jkl}$ components of the generalized second fundamental form of a curvature varifold $V$. Observe that $\sff_{jj}^k=P_{jl}A_{jlk}=A_{jjk}-P_{kl}A_{jjl}=A_{jkj}-P_{kl}A_{jlj}=nH_k-nP_{kl}H_l=nH_k$, and $A_{ijk}=\sff^k_{ij}+\sff^j_{ki}$.\\

\noindent In conclusion we state the compactness theorem that we use in this work.

\begin{thm}[\cite{Ma}]\label{thm:ConvergenceVarifolds}
	Let $p>1$ and $V_l$ a sequence of curvature varifolds with boundary in $\Om$. Call $A_{ijk}^{(l)}$ the functions $A_{ijk}$ of $V_l$. Suppose that $A_{ijk}^{(l)}\in L^p(V)$ and
	\begin{equation}
	\sup_l \quad\bigg\{\mu_{V_l}(W) + \int_{G_n(W)} \bigg|\sum_{i,j,k} |A^{(l)}_{ijk}|\bigg|^p\,dV_l + |\pa V_l|(G_n(W))\bigg\}\quad\le C(W)<+\infty
	\end{equation}
	for any $W\con\con G_n(\Om)$, where $|\pa V_l|$ is the total variation measure of $\pa V_l$. Then:\\
	i) up to subsequence $V_l$ converges to a curvature varifold with boundary $V$ in the sense of varifolds. Moreover $A^{(l)}_{ijk} V_l\to A_{ijk}V$ and $\pa V_l\to \pa V$ weakly* as measures on $G_n(\Om)$;\\
	ii) for every lower semicontinuous function $f:\R^{k^3}\to[0,+\infty]$ it holds that
	\begin{equation}
	\int_{G_n(\Om)} f(A_{ijk})\,dV \le \liminf_l \int_{G_n(\Om)} f(A_{ijk}^{(l)})\,dV_l.
	\end{equation}
\end{thm}

\noindent It follows from the above theorem that the Willmore energy is lower semicontinuous with respect to varifold convergence of curvature varifolds with boundary satisfying the hypotheses of Theorem \ref{thm:ConvergenceVarifolds}.\\
%

\bigskip

\section{Monotonicity formula and structure of varifolds with bounded energy}

\noindent The monotonicity formula on varifolds with locally bounded first variation is a fundamental identity proved in \cite{SiEX}, with important consequences on the structure of varifolds with bounded Willmore energy, collected for example in \cite{KuSc}. Such consequences usually concern varifolds without generalized boundary: $\si_V=0$. So, in this section we are interested in extending some of these results in the case of curvature varifold with boundary. The strategy is analogous to the one of \cite{KuSc} and the following results are probably expected by the experts in the field, however we prove them here for the convenience of the reader.\\
Let $V=\bv(M,\te_V)$ be a $2$-dimensional curvature varifold with boundary with finite Willmore energy. Denote by $\si_V$ the generalized boundary. Let $0<\si<\ro$ and $p_0\in\R^3$. Integrating the tangential divergence of the field $X(p)=\left( \frac{1}{|p-p_0|^2_\si} - \frac{1}{\ro^2} \right)_+(p-p_0)$, where $|p-p_0|^2_\si=\max\{\si^2,|p-p_0|^2\}$, with respect to the measure $\mu_V$ (see also \cite{SiEX} and \cite{RiLI}) one gets that
\begin{equation} \label{monot2}
A(\si)+\int_{B_\ro(p_0)\sm B_\si (p_0) } \bigg| \frac{\vec{H}}{2} +\frac{(p-p_0)^\perp}{|p-p_0|^2} \bigg|^2\,d\mu_V(p) = A(\ro),
\end{equation}
where
\begin{equation}
A(\ro):= \frac{\mu_V(B_\ro(p_0))}{\ro^2}+\frac{1}{4}\int_{B_\ro(p_0)} |H|^2\,d\mu_V(p)+ R_{p_0,\ro},
\end{equation}
and
\begin{equation}
\begin{split}
R_{p_0,\ro}&:= \int_{B_\ro(p_0)} \frac{\lgl \vec{H}, p-p_0\rgl}{\ro^2}\,d\mu_V(p) + \frac{1}{2}\int_{B_\ro(p_0)} \bigg( \frac{1}{|p-p_0|^2}-\frac{1}{\ro^2} \bigg)(p-p_0) \,d\si_V(p)\\& =: \int_{B_\ro(p_0)} \frac{\lgl \vec{H}, p-p_0\rgl}{\ro^2}\,d\mu_V(p) + T_{p_0,\ro}.
\end{split}
\end{equation}
In particular the function $\ro\mapsto A(\ro)$ is non-decreasing.

\noindent From now on, let us assume that the support $\supp\si_V\con S$, where $S$ is compact and $|\si_V|(S)<+\infty$. We also assume that
\[
\limsup_{R\to\infty} \frac{\mu_V(B_R(0))}{R^2}\le K<+\infty.
\]


\noindent We have that
\begin{equation}\label{eq7}
\begin{aligned}
\left| \int_{B_\ro(p_0)} \frac{\lgl \vec{H}, p-p_0\rgl}{\ro^2}\,d\mu_V(p)\right| &\le \left(\frac{\mu_V(B_\ro(p_0))}{\ro^2}\right)^{\frac12}\left(\int_{B_\ro(p_0)}|H|^2\,d\mu_V\right)^{\frac12}  \\
&\le \frac\ep2 \frac{\mu_V(B_\ro(p_0))}{\ro^2} + \frac2\ep \int_{B_\ro(p_0)}|H|^2\,d\mu_V.
\end{aligned}
\end{equation}
If $d(p_0,S)\ge\de$ we have that
\begin{equation}
\begin{aligned}
	\left| \int_{B_\ro(p_0)} \bigg( \frac{1}{|p-p_0|^2}-\frac{1}{\ro^2} \bigg)(p-p_0) \,d\si_V(p) \right|\le \left(\frac1\de +\frac1\ro\right) |\si_V|(S\cap B_\ro(p_0)).
\end{aligned}
\end{equation}
In particular the monotone function $A(\ro)$ evaluated at $p_0\not \in S$ is bounded below and there exists finite the limit $\lim_{\ro\searrow0} A(\ro)$. \\
Keeping $p_0\not\in S$ \eqref{monot2} implies that
\begin{equation}\label{eq4}
\begin{aligned}
	\frac{\mu_V(B_\si(p_0))}{\si^2}  &\le \frac{\mu_V(B_\ro(p_0))}{\ro^2}+\frac{1}{4}\int_{B_\ro(p_0)} |H|^2\,d\mu_V(p)+ R_{p_0,\ro} - R_{p_0,\si} \\
	&\le \frac{\mu_V(B_\ro(p_0))}{\ro^2} +\frac14\cW(V) + \left(\frac{\mu_V(B_\ro(p_0))}{\ro^2}\right)^{\frac12}\cW(V)^{\frac12} - T_{p_0,\si} +\left(\frac1\de +\frac1\ro\right) |\si_V|(S\cap B_\ro(p_0)) \\&\quad+  \frac\ep2 \frac{\mu_V(B_\si(p_0))}{\si^2} + \frac2\ep \cW(V) 
\end{aligned}
\end{equation}
\noindent Letting $\ro\to\infty$ and $\si<\de$ in \eqref{eq4} we get that $T_{p_0,\si}=0$ and
\begin{equation}\label{eq6}
	\frac{\mu_V(B_\si(p_0))}{\si^2}  \le C(\de,K,\cW(V)) <+\infty \qquad\quad\forall\,0<\si<\de,
\end{equation}
Letting $\ro\to0$ in \eqref{eq7} and using \eqref{eq6} we get that
\begin{equation}\label{eq8}
	\lim_{\ro\to0} \left| \int_{B_\ro(p_0)} \frac{\lgl \vec{H}, p-p_0\rgl}{\ro^2}\,d\mu_V(p)\right| =0. 
\end{equation}
Therefore we see that if $p_0\in\R^3\sm S$, then
\begin{equation}
	\exists\,\lim_{\si\searrow0} \frac{\mu_V(B_\si(p_0))}{\si^2}=\pi\te_V(p_0) \le C(\de,|\si_V|(S),K,\cW(V)).
\end{equation}
Moreover, consider $p_0\in\R^3\sm S$ and a sequence $p_k\to p_0$; let $\ro\in(0,d(p_0,S)/2)$ and call $\ro_0=d(p_0,S)/2$, then by \eqref{monot2} we have that
\begin{equation}
\begin{aligned}
	\frac{\mu_V(\overline{B_\ro(p_0)})}{\ro^2} &\ge \limsup_k \frac{\mu_V(B_\ro(p_k))}{\ro^2} \ge \limsup_k \pi\te_V(p_k) - R_{p_k,\ro} -\frac14\int_{B_\ro(p_k)} |H|^2\,d\mu_V \\
	&\ge \limsup_k \pi\te_V(p_k) - \int_{B_{2\ro}(p_0)} \frac{|H|}{\ro} \,d\mu_V -\frac14\int_{B_{2\ro}(p_k)} |H|^2\,d\mu_V  \\
	&\ge \limsup_k \pi\te_V(p_k) - \left(\frac{\mu_V(B_{2\ro}(p_0))}{\ro^2}\right)^{\frac12}\left(\int_{B_{2\ro}(p_0)}|H|^2\,d\mu_V\right)^{\frac12} -\frac14\int_{B_{2\ro}(p_k)} |H|^2\,d\mu_V  \\
	&\ge \limsup_k \pi\te_V(p_k) - \left(C(2\ro_0,|\si_V|(S),K,\cW(V))+\frac14\right)\left(\int_{B_{2\ro}(p_0)}|H|^2\,d\mu_V\right)^{\frac12},
\end{aligned}
\end{equation}
and thus letting $\ro\searrow0$ suitably we get
\begin{equation}
	\te_V(p_0)\ge \limsup_k \te_V(p_k),
\end{equation}
i.e. the multiplicity function $\te_V$ is upper semicontinuous on $\R^3\sm S$. Since $\te_V$ is integer valued, the set $\{p\in\R^3\sm S\,|\, \te_v(p)\ge\frac12\}$ is closed in $\R^3\sm S$. Therefore we can take the closed set $M=\{p\in\R^3\sm S\,|\, \te_v(p)\ge\frac12\}\cup S$ as the support of $V$.

\noindent A particular case of our analysis can be summarized in the following statement.

\begin{prop}\label{prop:StructureProperties}
	Let $V$ be a $2$-dimensional integer rectifiable curvature varifold with boundary. Denote by $\si_V$ the generalized boundary and by $S$ a compact set containing the support $\supp\si_V$. Assume that
	\[
	\cW(V)<+\infty, \quad \limsup_{R\to\infty} \frac{\mu_V(B_R(0))}{R^2}\le K<+\infty,
	\]
	and $S$ is a compact $1$-dimensional manifold with $\cH^1(S)<+\infty$. Then the limit
	\[
	\lim_{\ro\searrow0} \frac{\mu_V(B_\ro(p))}{\ro^2}
	\]
	exists at any point $p\in\R^3\sm S$, the multiplicity function $\te_V(p)=\lim_{\ro\searrow0} \frac{\mu_V(B_\ro(p))}{\ro^2}$ is upper semicontinuous on $\R^3\sm S$ and bounded by a constant $C(d(p,S),|\si_V|(S),K,\cW(V))$ depending only on the distance $d(p,S)$, $|\si_V|(S)$, $K$ and $\cW(V)$. Moreover $V=\bv(M,\te_V)$ where $M=\{p\in\R^3\sm S\,|\, \te_v(p)\ge\frac12\}\cup S$ is closed.
\end{prop}

\noindent Also, we can derive the following consequence.

\begin{cor}\label{cor:StimaUnboundedSupport}
	Let $V=\bv(M,\te_V)$ be a $2$-dimensional integer rectifiable curvature varifold with boundary with $\cW(V)<+\infty$. Denote by $\si_V$ the generalized boundary and by $S$ a compact set containing the support $\supp\si_V$. Assume that $S$ is a compact $1$-dimensional manifold with $\cH^1(S)<+\infty$. Then
	\begin{equation}
		M \mbox{ ess. unbounded} \qquad \Leftrightarrow \qquad \limsup_{\ro\to\infty} \frac{\mu_V(B_\ro(0))}{\ro^2} \ge \pi,
	\end{equation}
	where $M$ essentially unbounded means that for every $R>0$ there is $B_r(x)\con \R^3\sm B_R(0)$ such that $\mu_V(B_r(x))>0$.\\
	Moreover, in any of the above cases the limit $\lim_{\ro\to\infty} \frac{\mu_V(B_\ro(0))}{\ro^2} \ge\pi$ exists.
\end{cor}

\begin{proof}
	Suppose that $M$ is essentially unbounded. We can assume that $\limsup_{\ro\to\infty} \frac{\mu_V(B_\ro(0))}{\ro^2} \le K<+\infty$. Then
	\[
	\begin{split}
			\left| \int_{B_\ro(0)} \frac{1}{\ro^2} \lgl \vec{H}, p \rgl \,d\mu_V \right|&\le \frac{1}{\ro^2} \left( \int_{B_\si(0)} |H||p|\,d\mu_V(p) + \int_{B_\ro(0)\sm B\si(0)} |H||p|\,d\mu_V(p) \right) \\&\le \frac{\si}{\ro^2} \sqrt{\int_{B_\si(0)} |H|^2\,d\mu_V }\sqrt{\mu_V(B_\si(0))} + \sqrt{\frac{\mu_V(B_\ro(0))}{\ro^2}}\sqrt{\int_{B_\ro(0)\sm B_\si(0)} |H|^2\,d\mu_V}
	\end{split}
	\]
	for any $0<\si<\ro<+\infty$. Passing to the $\limsup_{\ro\to\infty}$ and then to $\si\to\infty$, we conclude that
	\[
	\lim_{\ro\to\infty} \left| \int_{B_\ro(0)} \frac{1}{\ro^2} \lgl \vec{H}, p \rgl \,d\mu_V \right|=0.
	\]
	Hence, assuming without loss of generality that $0\not\in S$, the monotone quantity $A(\ro)$ evaluated on $V$ with base point $0$ gives
	\[
	\exists\,\lim_{\ro\to\infty} A(\ro)= \cW(V) + \frac12 \int \frac{p}{|p|^2}\,d\si_V(p) + \limsup_{\ro\to\infty} \frac{\mu_V(B_\ro(0))}{\ro^2}, 
	\]
	and thus $\exists\,\lim_{\ro\to\infty}  \frac{\mu_V(B_\ro(0))}{\ro^2} \le K<+\infty$. Also the assumptions of Proposition \ref{prop:StructureProperties} are satisfied and we can assume that $M$ is closed.
	
	We can prove that $M$ has at least one unbounded connected component. In fact any compact connected component $N$ of $M$ defines a varifold $\bv(N,\te_V|_N)$ with generalized mean curvature; now if $S\cap N=\epty$ then $\cW(N)\ge4\pi$, and thus there are finitely many compact connected components without boundary, if instead $S\cap N\neq\emptyset$, $S\con B_{R_0}(0)$ by compactness, and $\exists\,p_0\in N\sm B_r(0)$ for $r>R_0$ but $N$ is compact, then the monotonicity formula applied on $\bv(N,\te_V|_N)$ at point $p_0$ gives
	\begin{equation}\label{eq15}
	\pi\le \lim_{\si\to0} A_{\bv(N,\te_V|_N)}(\si) \le \lim_{\ro\to\infty}  A_{\bv(N,\te_V|_N)}(\ro) \le \frac14\cW(\bv(N,\te_V|_N)) +\frac12 \frac{|\si_V|(S)}{r-R_0}.
	\end{equation}
	Since $M$ is essentially unbounded, if any connected component of $M$ is compact we would find infinitely many compact connected components $N$, points $p_0\in N$, and $r$ arbitrarily big in \eqref{eq15} so that the Willmore energy of any such $N$ is greater than $2\pi$, implying that $\cW(V)=+\infty$.\\	
	As $M$ has a connected unbounded component, for any $\ro$ sufficiently large there is $x_\ro\in M\cap B_{2\ro}(0)$. Applying the monotonicity formula on $V$ at $x_\ro$ for $\ro$ sufficiently big so that $S\con B_\ro(0)$ we get that
	\[
	\begin{split}
		\pi \le \lim_{\si\to0} A(\si) &\le \frac{\mu_V(B_\ro(x_\ro))}{\ro^2} + \frac14 \int_{B_\ro(x_\ro)} |H|^2\,d\mu_V + \frac{1}{\ro}\int_{B_\ro(x_\ro)} |H|\,d\mu_V \\
		&\le  9\frac{\mu_V(B_{3\ro}(0))}{(3\ro)^2} +\frac14 \int_{\R^3\sm B_\ro(0)} |H|^2\,d\mu_V + \ep \frac{\mu_V(B_\ro(x_\ro))}{\ro^2}  + C_\ep  \int_{B_\ro(x_\ro)} |H|^2\,d\mu_V ,
	\end{split}
	\]
	that implies that
	\[
	\lim_{\ro\to\infty} \frac{\mu_V(B_\ro(0))}{\ro^2} \ge \frac{\pi}{9+\ep},
	\]
	for any $\ep>0$.\\	
	Consider now any sequence $R_n\to\infty$ and the sequence of blow-in varifolds given by
	\[
	V_n=\bv\left(\frac{M}{R_n}, \te_n\right),
	\] 
	where $\te_n(x)=\te_V(R_nx)$. Since
	\[
	\mu_{V_n}(B_R(0)) =\frac{1}{R_n^2}\mu_V(B_{R_nR}(0))=\frac{1}{(RR_n)^2}\mu_V(B_{RR_n}(0))R^2\le K'R^2
	\]
	is bounded for any $R>0$, $\cW(V_n)=\cW(V)$, and $|\si_{V_n}|(\R^3)\to0$, by the classical compactness theorem of rectifiable varifolds (Theorem 42.7 in \cite{SiGMT}) we get that $V_n$ converges to an integer rectifiable varifold $W$ (up to subsequence). Also $W\neq0$, in fact $0\in\supp W$ by the fact that
	\[
	\mu_W(\overline{B_1(0)}) \ge \liminf_n \mu_{V_n} (B_1(0))= \liminf_n \frac{\mu_V(B_{R_n}(0))}{R_n^2}\ge \frac\pi9.
	\]
	We have that $W$ is stationary, in fact for any $r>0$ we have that
	\[
	\begin{split}
		\int_{\R^3\sm\overline{B_r(0)}} |H_W|^2\,d\mu_W \le \liminf_n \int_{\R^3\sm\overline{B_r(0)}} |H_{V_n}|^2\,d\mu_{V_n} = \liminf_n \int_{\R^3\sm\overline{B_{R_nr}(0)}} |H_V|^2\,d\mu_V=0. 
	\end{split}
	\]
	Also $\si_W=0$, in fact for any $X\in C^0_c(\R^3)$ the convergence of the first variation reads
	\[
	\lim_n -2 \int \lgl H_{V_n}, X \rgl\,d\mu_{V_n} +  \int X\,d\si_{V_n} = \lim_n -2 \int \lgl H_{V_n}, X \rgl\,d\mu_{V_n} = \int X\,d \si_V,
	\]
	and $\supp\si_V\con\{0\}$. Taking $X=\La_m Y$ for $Y\in C^0_c(\R^3)$ and
	\[
	\La_m(p)=\begin{cases}
	1-md(p,0) & d(p,0)\le\frac1m,\\
	0 & d(p,0)>\frac1m,
	\end{cases}
	\]
	we see that
	\[
	\left| \int \lgl H_{V_n}, X \rgl\,d\mu_{V_n} \right| = \left| \int_{B_{\frac1m}(0)} \lgl H_{V_n} , \La_m Y  \rgl \, d\mu_{V_n} \right| \le \|Y\|_\infty \cW(V)^{\frac12}\left(K'\frac{1}{m^2}\right)^{\frac12},
	\]
	and thus
	\[
	\int Y\,d\si_V = \lim_n -2 \int \lgl H_{V_n}, \La_m Y \rgl\,d\mu_{V_n} =\lim_{m\to\infty}  \lim_n -2 \int \lgl H_{V_n}, \La_m Y \rgl\,d\mu_{V_n} =0,
	\]
	for any $Y\in C^0_c(\R^3)$.\\	
	Finally the monotonicity formula applied on $W$ gives
	\[
	\lim_n \frac{\mu_V(R_n(0))}{R_n^2}\ge\liminf_n \mu_{V_n}(B_1(0))\ge\mu_W(B_1(0))\ge \lim_{\si\to0} A_W(\si) \ge \pi.
	\]
\end{proof}


\end{document}